\documentclass[12pt]{article}
\usepackage{authblk}
\usepackage[english]{babel}
\usepackage[reqno]{amsmath}
\usepackage{amssymb}
\usepackage[T1]{fontenc}
\usepackage{lmodern}
\usepackage{amsthm}
\usepackage{hyperref}
\usepackage{graphics, pict2e}
\usepackage{enumitem}
\usepackage{fancyhdr}
\usepackage{xcolor}
\usepackage{bbm}
\usepackage[margin=0.99in]{geometry}
\usepackage{mathtools}
\mathtoolsset{showonlyrefs=true}
\numberwithin{equation}{section}

\newtheorem{proposition}{Proposition}[section]
\newtheorem{theorem}{Theorem}[section]
\newtheorem{corollary}{Corollary}[section]
\newtheorem{lemma}{Lemma}[section]
\theoremstyle{definition}
\newtheorem{remark}{Remark}[section]

\newtheorem{example}{Example}[section]

\makeatletter
\newcommand{\bigcomp}{%
  \DOTSB
  \mathop{\vphantom{\sum}\mathpalette\bigcomp@\relax}%
  \slimits@
}
\newcommand{\bigcomp@}[2]{%
  \begingroup\m@th
  \sbox\z@{$#1\sum$}%
  \setlength{\unitlength}{0.9\dimexpr\ht\z@+\dp\z@}%
  \vcenter{\hbox{%
    \begin{picture}(1,1)
    \bigcomp@linethickness{#1}
    \put(0.5,0.5){\circle{1}}
    \end{picture}%
  }}%
  \endgroup
}
\newcommand{\bigcomp@linethickness}[1]{%
  \linethickness{%
      \ifx#1\displaystyle 2\fontdimen8\textfont\else
      \ifx#1\textstyle 1.65\fontdimen8\textfont\else
      \ifx#1\scriptstyle 1.65\fontdimen8\scriptfont\else
      1.65\fontdimen8\scriptscriptfont\fi\fi\fi 3
  }%
}

\newcommand{\n}{\mathbb{N}}
  
\renewcommand {\>}{\right\rangle}  
\newcommand {\norma}[1]{\left\|#1\right\|}
\newcommand {\onorma}[1]{\left\|#1\right\|_{\mathrm{op}}} 

\newcommand{\ew}{\mathbb{E}}
\newcommand{\pr}{\mathbb{P}}
\newcommand{\bmathcal}[1]{\bar{\mathcal{#1}}}

\newcommand{\A}[2]{\mathbb{A}_{#1}(#2)}
\newcommand{\wA}{\widehat{\mathbb{A}}}

\newcommand{\emails}[1]{%
  \begingroup
  \renewcommand\thefootnote{}%
  \footnotetext{#1}%
  \endgroup
}

\title{On the Existence of Geometrically Attracting Measures for Iterated Function Systems with Varying Sets of Transformations}

\author[1]{Dawid Czapla\footnote{Corresponding author}}
\author[2]{Rafa\l{} Kapica}
\author[1]{Maciej \'Sl\k{e}czka}

\affil[1]{Institute of Mathematics, University of Silesia in Katowice, Bankowa 14, Katowice 40-007, Poland}

\affil[2]{Faculty of Applied Mathematics, AGH University of Science and Technology, A. Mickiewicza 30, Krak\'ow 30-059, Poland}

\date{}
\begin{document}
\maketitle
\emails{\emph{Email addresses:} \texttt{dawid.czapla@us.edu.pl} (D. Czapla),\, \texttt{rafal.kapica@agh.edu.pl} (R. Kapica),\\ \texttt{maciej.sleczka@us.edu.pl} (M. \'Sl\k{e}czka)}

\vspace{-1cm}
\begin{abstract}
In this paper, we focus on the asymptotic behavior of random iterated functions systems, in which both the family of transformations and the distribution of selecting them vary at each step. We consider two types of dynamics: the time-inhomogeneous Markov chain arising from forward iterates and the non-Markovian process generated by backward iterates. For both settings, we provide certain criteria for the existence of a~probability measure that is geometrically attracting in the bounded Lipschitz distance, independently of the initial distribution. Finally, we illustrate our results through applications to specific models involving affine transformations.
\end{abstract}
{\small
}

\section*{Introduction}
From the probabilistic perspective, an \emph{iterated function system} (IFS) is usually understood as a framework for modelling a discrete-time Markov process through the random selection and successive composition of transformations within some fixed family. 

Early precursors of probabilistic IFSs can be already found in the theory of \emph{systems with complete connections}, introduced by Mihoc and Onicescu in the 1930s (\cite{b:MihocOnicescu1935}), as well as in mathematical models of learning (see, e.g., \cite{b:Karlin1953, b:Norman1968}). A major contribution, however, was made by Dubins and Freedman, who studied the existence and uniqueness of stationary distributions for Markov chains generated by the successive random application of transformations (\cite{b:DubinsFreedman1966}), thereby establishing one of the principal probabilistic foundations of what is now called the theory of iterated random functions. Independently, the geometric foundations of the subject were laid by Hutchinson, who proved that a finite family of contractions on a complete metric space determines a unique compact self-similar invariant set (called an \emph{attractor}), and, when the transformations are assigned fixed selection probabilities, a unique invariant probability measure supported on that set (\cite{b:Hutchinson1981}). The notion of an IFS was later introduced and popularized by Barnsley and Demko through their systematic approach to the construction and approximation of fractals (\cite{b:BarnsleyDemko1985}). Soon afterwards, the framework was substantially extended by allowing the probabilities of selecting the maps to depend on the current state of the process (\cite{b:BarnsleyDemkoEltonGeronimo1988}). This development gave rise to a general Markov-operator approach, in which questions concerning invariant measures, asymptotic stability (i.e., existence of a unique, weakly attractive invariant distribution), and rates of convergence to equilibrium became central. Incidentally, the concept of asymptotic stability was used to generalize classical self-similar fractals to the so-called \emph{semifractals}, introduced by Lasota and Myjak in \cite{b:LasotaMyjak1998}. Results concerning the above questions were first established mainly for compact or locally compact state spaces (see, e.g., \cite{b:Lasota1995, b:LasotaYorke1994, b:Stenflo2002, b:Werner2005}), and subsequently extended to Polish spaces (see, e.g., \cite{b:KapicaSleczka2020,b:Stenflo2001, b:Szarek2003IM, b:Szarek2003}). 

In the modern approach, given a Polish metric space~$X$ (viewed as a state space) and a~measurable space $(\Theta,\mathcal{A})$, an~IFS is typically defined as a~system consisting of a collection \hbox{$\{\vartheta_x:\; x\in X\}$} of probability measures on $\mathcal{A}$ and a family $\{S_{\theta}:\; \theta\in\Theta\}$ of transformations of $X$ into itself (usually assumed to be continuous), such that $(x,\theta)\mapsto S_{\theta}(x)$ is measurable. The dynamics conventionally associated with this framework coincide with the movement obtained from the so-called \emph{forward iterates}. More precisely, they can be described by the random process $\mathbb{X}:=\{X_n\}_{n\in\n_0}$ (with $\mathbb{N}_0:=\mathbb{N}\cup\{0\}$) defined as
$$X_n:=\left(S_{\eta_n}\circ\ldots\circ S_{\eta_1}\right)(X_0)=S_{\eta_n}(X_{n-1})\quad\text{for}\quad n\in\n,$$
where $\{\eta_n\}_{n\in\n}$ is a sequence of $\Theta$-valued random variables such that, for any $x\in X$ and~$n\in\n$, the conditional distribution of $\eta_n$ given $X_{n-1}=x$ is $\vartheta_x$. The process $\mathbb{X}$ naturally constitutes a time-homogeneous Markov chain evolving on $X$. Crucially, many Markov chains occurring in the natural sciences can be represented in this manner. Among numerous examples, one could mention models for the intracellular biochemistry of a generic cell undergoing mitosis (\cite{b:LasotaMackey1999, b:Wojewodka2013}) and the growth of blood cell population (\cite{b:Wazewska}; cf. also \cite[\S7.3]{b:Szarek2003} for the relevant stability result).

Beyond the above, one may also be concerned with the dynamics given by the \emph{backward iterates}, represented by the reverse process $\mathbb{Y}:=\{Y_n\}_{n\in\n_0}$ of the form
$$Y_n:=\left(S_{\eta_1}\circ\ldots\circ S_{\eta_n}\right)(Y_0)\quad\text{for}\quad n\in\n,$$
which, in general, is not Markovian. The rationale for considering such dynamics is twofold. On the one hand, the process $\mathbb{Y}$ is itself applicable to modeling specific phenomena, such as perpetuities (see e.g. \cite{b:AlsmeyerIksanovRosler2009,b:EmbrechtsKluppelbergMikosch1997, b:GoldieMaller2000, b:Iksanov2016}), where the $S_{\theta}$ are affine mappings on~$\mathbb{R}^d$. On the other hand, setting $Y_0=X_0$, it becomes an effective tool for investigating the chain~$\mathbb{X}$ when the transformations are sampled independently of the current state. More precisely, if $\vartheta_x=\vartheta_y$ for any $x,y\in X$, which in practice amounts to assuming that $\eta_1,\eta_2,\ldots$ are mutually independent and identically distributed, with $\{\eta_n\}_{n\in\n}$ being independent of $X_0$, then $X_n$ and~$Y_n$ share the same distribution for every~$n$. This observation is particularly useful for studying the asymptotic behavior of~$\mathbb{X}$, since the convergence in law of $\{Y_n\}_{n\in\n_0}$ is often easier to establish than that of $\{X_n\}_{n\in\n_0}$. For instance, it plays a key role in the proof of \hbox{\cite[Theorem 5.1]{b:DiaconisFreedman1999}}, which -- under a \emph{contractivity-on-average}-type condition and certain additional assumptions --  ensures that the law of the chain $\mathbb{X}$ converges at an exponential rate in the L\'evy--Prokhorov metric to its unique stationary distribution, regardless of the initial state.

In this paper, we study a generalization of classical IFSs (as defined above) in which transformations are selected according to a fixed, state-independent distribution. Specifically, inspired by  \cite{b:Stenflo1998} and \cite{b:Mendivil2015}, our setting allows both the family of admissible transformations and the probability law governing their selection to vary with the iteration step. Thus, rather than working with a single IFS $((\Theta, \mathcal{A}, \vartheta),\, \{S_{\theta}:\; \theta\in\Theta\})$, we consider a sequence $((\Theta_n, \mathcal{A}_n, \vartheta_n),\, \{S_{\theta}^{(n)}:\; \theta\in\Theta_n\})_{n\in\n}$ of such systems, called a \emph{time-dependent IFS}. Within this framework, the forward and backward dynamics are represented by the processes~\hbox{$\Phi=\{\Phi_n\}_{n\in\n_0}$} and $\Psi=\{\Psi_n\}_{n\in\n_0}$, respectively, whose $n$-th terms are defined analogously to $X_n$ and~$Y_n$, respectively, except that $S_{\eta_k}$ is replaced by the step-dependent transformation $S_{\eta_k}^{(k)}$ for each $k\in\{1,\ldots,n\}$. The process $\Phi$ is therefore a time-inhomogeneous Markov chain, which makes the analysis of its asymptotic behavior considerably more challenging, since most of techniques developed for homogeneous chains are no longer applicable. In particular, despite the transformations being sampled independently of the current state, the method based on the asymptotic behavior of the reversed process~$\Psi$ fails either. 


Our main goal is to identify tractable conditions under which, for any initial distribution with a finite first moment, the law of the chain~$\Phi$ converges at a geometric rate in the \emph{bounded Lipschitz distance} (also referred to as the Fortet--Mourier or Dudley metric; see \hbox{\cite{b:Dudley1966, b:FortetMourier1953, b:Szarek2003}}) to a~Borel probability measure, and to carry out a~similar analysis for the process~$\Psi$. These objectives are addressed in Theorems \ref{thm:main} and \ref{thm:main2}, respectively, and are supplemented by Propositions \ref{prop:finite_moment} and~\ref{cor:finite_moment}, which introduce additional assumptions ensuring that the limiting measures have finite first moments. Notably, the most recent general results of this kind for classical IFSs, including systems with state-dependent selection distributions, appear to be \hbox{\cite[Proposition~3.1]{b:KapicaSleczka2020} and \cite[Theorem 1]{b:Stenflo2001}}.

As is well known, the bounded Lipschitz distance metrizes the topology of weak convergence on the space of Borel probability measures (in fact, it is equivalent to the Prokhorov metric). The aforementioned results therefore imply, in particular, that the laws of $\Phi$ and~$\Psi$ converge weakly, independently of the initial distribution. Nevertheless, even when the transition operators of~$\Phi$ are Feller -- as is the case under our assumptions -- one should not generally expect the limiting distributions of $\Phi$ or~$\Psi$ to be stationary (see Remark~\ref{rem:1} and Example~\ref{rem:2.5}), in contrast to the time-homogeneous setting. Determining their invariance constitutes a separate and complex problem that lies beyond the scope of the present study.

Our approach basically rests on two core assumptions. The first expresses a form of contractivity on average in the logarithmic sense. Specifically, each random map $x\mapsto S_{\eta_n}^{(n)}(x)$ is required to be almost surely Lipschitz continuous, and the corresponding random Lipschitz constants~$L_n(\eta_n)$ are assumed to be identically distributed with a negative expected logarithm. The second, in turn, provides a uniform integrability bound on the one-step displacements that enter the estimates for successive compositions. As might be expected, this control takes a much more involved form in the case of the forward dynamics of $\Phi$;  namely, it concerns conjugated displacements, relying on the additional requirement that~$S_{\theta}^{(n)}$ are invertible. For the backward process~$\Psi$, direct one-step displacements suffice, and no invertibility is required. The proofs of our main results rely primarily on \hbox{\cite[Lemma 5.2]{b:DiaconisFreedman1999}} by Diaconis and Freedman (stated here as Lemma~\ref{lem:diaconis}) and Lemma~\ref{lem:composition}. The former ensures that successive partial products of the random Lipschitz constants (with a negative mean logarithm) decay geometrically outside events of exponentially small probability, whereas the latter bounds the distance between consecutive compositions in terms of the relevant Lipschitz factors and displacement terms.

When comparing Theorem \ref{thm:main} with \cite[Theorem 2.1]{b:Stenflo1998} of Stenflo, it should be stressed that although their conclusions largely overlap, they accommodate time-inhomogeneity in fundamentally different ways. Stenflo treats the time-dependent system as an asymptotically vanishing perturbation of a fixed, contractive-on-average IFS whose associated Markov chain is exponentially ergodic. At each step, the time-dependent and respective reference maps are coupled so that their expected discrepancy vanishes uniformly over the state space and satisfies a~weighted summability condition that determines the convergence rate. By contrast, our hypotheses refer directly to step-dependent IFSs and do not presuppose the existence of any limiting time-homogeneous system. This treatment is conceptually closer to that of Mendivil in \cite{b:Mendivil2015}, who studies time-dependent IFSs composed of finitely many non-expansive transformations that converge uniformly to the identity as their Lipschitz constants approach one, with a focus on certain affine maps on~$\mathbb{R}^{d}$. Notably, while our results are not intended to generalize those in \cite{b:Mendivil2015}, the framework developed here is substantially broader, as it permits arbitrary, potentially uncountable families of not necessarily affine transformations.

In the context of affine transformations, another contribution of this work is the adaptation of Theorems~\ref{thm:main} and \ref{thm:main2} to the case where $S_{\theta }^{(n)}$ are maps of this kind on a Banach space. The results obtained by specializing our assumptions to that scenario are stated in Propositions~\ref{prop:affine_general} and \ref{prop:affine_general2}, which concern the processes $\Phi$ and $\Psi$, respectively. Moreover, in Example~\ref{example:1}, we introduce a fairly broad subclass of affine transformations to which Proposition~\ref{prop:affine_general} can be applied by verifying the convergence of certain explicitly defined numerical series. It is noteworthy here that affine IFSs are widely used in the theory of stochastic recurrence equations (see~\cite{b:BuraczewskiDamekMikosch2016, b:Vervaat1979}), which, in turn, find applications in fields such as finance and insurance (including the study of the aforementioned perpetuities \cite{b:GoldieMaller2000}), as well as telecommunications, physics (e.g., nuclear technology), biology, and time-series analysis.

The outline of the paper is as follows. Section~\ref{sec:1} introduces a general framework for time-inhomogeneous Markov chains, including the relevant notation and basic concepts concerning Markov operators, as well as a general criterion for the finiteness of the first moment of an attracting measure (Proposition~\ref{prop:m_11}). Sections~\ref{sec:2} and~\ref{sec:3} contain formal descriptions of the models arising in our setting from the forward and backward iterates, respectively, and state the corresponding main results  (Theorems~\ref{thm:main} and~\ref{thm:main2}). In addition, Section~\ref{sec:2} features an example demonstrating that the geometrically attracting probability measure yielded by Theorem \ref{thm:main} need not be invariant. Section~\ref{sec:4} is devoted to adapting our general results to systems of affine transformations (Propositions~\ref{prop:affine_general} and~\ref{prop:affine_general2}) and provides several representative examples. The proofs of the main theorems, along with all necessary auxiliary results, are presented in Section~\ref{sec:5}. Finally, \hyperref[sec:appendix]{Appendix} offers a~detailed proof of Lemma~\ref{lem:diaconis}, filling in certain details omitted in \cite{b:DiaconisFreedman1999}.


\section{Preliminaries} \label{sec:1}
Let us begin with introducing a~piece of notation. Given a~metric space $(X,\rho)$, we shall write $\mathcal{B}(X)$ for its Borel $\sigma$-field. The symbol $B_b(X)$ will stand for the space of all bounded Borel measurable functions from $X$ to $\mathbb{R}$, while $C_b(X)$ and $L_{b,1}(X)$ will denote its subsets consisting of all continuous functions and all Lipschitz continuous functions with the minimal Lipschitz constant $\leq 1$ that take values in $[-1,1]$, respectively. Further, let $\mathcal{M}_1(X)$ be the family of all Borel probability measures on $X$, and let $\mathcal{M}_{1,1}(X)$ denote the subset of $\mathcal{M}_1(X)$ consisting of all measures with finite first moment, i.e., those $\mu\in\mathcal{M}_1(X)$ for which $\int_X \rho(x,x_0)\,\mu(dx)<\infty$ with some (and thus with all) $x_0\in X$. For notational brevity, given any finite (signed) Borel measure $\mu$ on~$X$ and any Borel measurable function $f:X\to\mathbb{R}$, we will write $\<f, \mu\>$ for $\int_X f\, d\mu$, whenever the integral makes sense. The indicator function of a~subset $A$ of any given set (depending on the context) will be denoted by $\mathbbm{1}_A$, and the Dirac measure at a~point $x$ will be labeled as $\delta_x$. Finally, we put $\mathbb{R}_+:=[0,\infty)$ and $\mathbb{N}_0:=\mathbb{N}\cup\{0\}$.

Let us recall that a~sequence $\{\mu_n\}_{n\in\n}$ of finite Borel measures on $X$ is said to be \emph{weakly convergent} to a~measure $\mu$ (which is written as $\mu\stackrel{w}{\to}\mu$) if $\lim_{n\to\infty} \<f,\mu_n\>=\<f,\mu\>$ for all $f\in C_b(X)$. It is well known (see \cite[Theorems 6 and 8]{b:Dudley1966}) that the topology on $\mathcal{M}_1(X)$ determined by the weak convergence of measures coincides with that induced by the aforementioned \emph{bounded Lipschitz distance}, provided that $X$ is a~Polish space. This distance can be defined~as
$$d_{BL}(\mu, \nu):=\sup_{f\in L_{b,1}(X)}|\<f,\mu-\nu\>|\quad\text{for all}\quad \mu,\nu\in\mathcal{M}_1(X).$$
Notably, if $X$ is Polish, then the space $(\mathcal{M}_1(X), d_{BL})$ is complete (see~\hbox{\cite[Theorems 9]{b:Dudley1966}}).

We shall now briefly review certain basic concepts related to Markov operators and time-inhomogeneous Markov chains. First of all, by a~stochastic kernel (on $X$) we mean a~function $P:X\times \mathcal{B}(X)\to[0,1]$ such that $X\ni x\mapsto P(x,A)$ is Borel measurable for each $A\in\mathcal{B}(X)$, and $\mathcal{B}(X)\ni A \mapsto P(x,A)$ is a~probability measure for every $x\in X$. The composition $PQ$ of any two stochastic kernels $P$ and $Q$ on $X$ is defined, as usual, by
\begin{equation}
\label{e:chapman}
PQ(x,A):=\int_X Q(y, A)\,P(x,dy)\quad\text{for}\quad x\in X,\;A\in\mathcal{B}(X).
\end{equation}
Consequently, the iterates of a~kernel $P$ are of the form
$$P^1:=P,\quad P^{n+1}(x,A)=\int_X P^n(y,A)\,P(x,dy)\quad \text{for}\quad  x\in X,\;A\in\mathcal{B}(X),\;n\in\n.$$

Given a~stochastic kernel $P$ on $X$, we will consider two associated operators (also denoted~$P$): one, $(\cdot)P$, acting from $\mathcal{M}_1(X)$ into itself, and the other, $P(\cdot)$, acting from $B_b(X)$ into itself, defined~by
$$
\mu P(A):=\int_X P(x,A)\,\mu(dx)\quad\text{for}\quad \mu\in\mathcal{M}_1(X),\;A\in\mathcal{B}(X),$$
$$
Pf(x):=\int_X f(y)\,P(x,dy)\quad\text{for}\quad f\in B_b(X),\; x\in X.
$$
The first of them will be referred to as a~\emph{Markov operator}, and the second as its dual. These two operators are linked by the standard duality relation
$$\<f,\mu P\>=\<Pf,\mu\>\quad\text{for any}\quad f\in B_b(X),\;\mu\in\mathcal{M}_1(X).$$
Naturally, $P(\cdot)$ can also be viewed as an operator defined on the larger class of Borel functions that are only bounded from below. For such functions $f$, the above relationship remains valid as well, but one should keep in mind that $Pf$ may take $\infty$ as a~value in this case. We mention this because, is at some point the paper, we will write $PV$ for $V:=\rho(x_0,\cdot)$.

A stochastic kernel $P$ (or the induced Markov operator) is called \emph{Feller} if $Pf\in C_b(X)$ for every $f\in C_b(X)$. It is easy to check that if $P$ is Feller, then the corresponding Markov operator on $\mathcal{M}_1(X)$ is continuous in the topology of weak convergence of measures. Let us also recall that a~measure $\mu_*\in\mathcal{M}_1(X)$ is said to be \emph{invariant} (or \emph{stationary}) for a~Markov operator $P$ whenever $\mu_* P=\mu_*$.

Now, let $\{P_n\}_{n\in\n}$ be a~sequence of stochastic kernels on $X$, and let \hbox{$\mu\in\mathcal{M}_1(X)$}. By a~Markov chain with initial distribution $\mu$ and transition probabilities $P_n$, $n\in\n$, we mean a~sequence \hbox{$\Phi^{\mu}:=\{\Phi_n^{\mu}\}_{n\in\mathbb{N}_0}$} (for~simplicity, sometimes written without the superscript~$\mu$) of $X$-valued random variables on some probability space $(\Omega,\mathcal{F},\pr)$ such that $\Phi_0^{\mu}\sim \mu$ (i.e., $\mu$ is the distribution of $\Phi_0^{\mu}$) and, for any $A\in\mathcal{B}(X)$ and $n\in\n$,
$$\pr(\Phi_n\in A\,|\,\Phi_0,\ldots,\Phi_{n-1})=\pr(\Phi_n\in A\,|\,\Phi_{n-1})=P_n(\Phi_{n-1},\, A)\;\;\text{a.s.}
$$
In this setting, $P_n$ is referred to as the \emph{one-step transition kernel} of $\Phi$ at time $n$. If $P_n=P_1$ for all $n\in\mathbb{N}$, then $\Phi$ is called \emph{time-homogenous}.

For a~Markov chain $\Phi$ as above, the $n$-step transition probability from a~state $x\in X$ at time $k$ to a~set $A\in\mathcal{B}(X)$ takes the form
$$\pr(\Phi_{k+n}\in A\,|\,\Phi_k=x)=P_{k+1}\ldots P_{k+n}(x,A),$$
where the right hand-side denotes the composition of kernels in the sense of \eqref{e:chapman}. Moreover, letting $\mu_n$ denote the distribution of $\Phi_n$ for every $n\in\mathbb{N}_0$ (so that $\mu_0=\mu$), we see that
\begin{align*}
\mu_n(A)&=\pr(\Phi_n\in A)=\ew\left[\pr(\Phi_n\in A\,|\,\Phi_{n-1}) \right]=\ew[P_n(\Phi_{n-1},\,A)]\\
&=\int_X P_n(x,A)\,\mu_{n-1}(dx)=\mu_{n-1} P_n(A) \quad\text{for all}\quad A\in\mathcal{B}(X),\; n\in\mathbb{N},
\end{align*}
which shows that $\mu_{k+n} =\mu_k P_{k+1}\ldots P_{k+n}$ for all $k\in\mathbb{N}_0$ and $n\in\mathbb{N}$. Hence, putting 
\begin{equation}
\label{e:composition}
P^{(n)}:=P_1\ldots P_n \quad\text{for}\quad n\in\mathbb{N},
\end{equation}
gives, in particular,
\begin{equation}
\label{e:trans_op}
\mu_n=\pr(\Phi_n^{\mu}\in \cdot)=\mu P^{(n)} \quad\text{for every}\quad n\in\mathbb{N}.
\end{equation}

In what follows, we will also deal with non-Markovian stochastic processes whose successive distributions are determined by acting the reversed compositions of Markov operators on the initial distribution.  Specifically, given a~sequence of stochastic kernels $\{P_n\}_{n\in\n}$ on~$X$ and $\mu\in\mathcal{M}_1(X)$, one can consider a~process $\Psi^{\mu}:=
\{\Psi_n^{\mu}\}_{n\in\n_0}$ (for simplicity, sometimes written without $\mu$)  of random variables with values in $X$ such that~$\Psi_0^{\mu}\sim\mu$, and
\begin{equation}\label{e:trans_op_rev}
\pr\left(\Psi_n^{\mu}\in \cdot\right)=\mu \bar{P}^{(n)}\quad\text{for every}\quad n\in\n,
\end{equation}
where 
\begin{equation}
\label{e:composition_rev}
\bar{P}^{(n)}:=P_n\ldots P_1 \quad\text{for}\quad n\in\mathbb{N}.
\end{equation}

In view \eqref{e:trans_op} (respectively, \eqref{e:trans_op_rev}), it is reasonable to call a~measure $\mu_*\in\mathcal{M}_1(X)$ a~\emph{stationary}
(or \emph{invariant}) \emph{distribution} for $\Phi$ (resp., $\Psi$) if \hbox{$\mu_* P^{(n)}=\mu_*$} (resp., \hbox{$\mu_* \bar{P}^{(n)}=\mu_*$}) for each $n\in\mathbb{N}$. In the case of $\Phi$ (though not necessarily of $\Psi$), this is plainly equivalent to requiring that $\mu_* P_n =\mu_*$ for all $n\in\mathbb{N}$. By saying that a~measure $\pi\in\mathcal{M}_1(X)$ is \emph{attracting} for $\Phi$ (resp., $\Psi)$ we mean, in turn, that 
\begin{equation}\label{def:attracting}
\mu P^{(n)}\stackrel{w}{\to} \pi\;\; \text{(resp.,}\;\;\mu \bar{P}^{(n)}\stackrel{w}{\to} \pi)\quad\text{for every}\quad \mu\in\mathcal{M}_1(X).
\end{equation}

\begin{remark}\label{rem:1}
It is clear that, for any Markov operator $P$ with the Feller property, the convergence $\mu P^n \stackrel{w}{\to}\mu_*$ for some $\mu,\mu_*\in\mathcal{M}_1(X)$ implies that $\mu_*$ is invariant for $P$. In particular, every attracting probability measure of a time-homogeneous Markov chain with a Feller transition law is stationary. This, however, need not be the case for inhomogeneous Markov chains or for processes governed by \eqref{e:trans_op_rev}. 

Indeed, suppose that $X$ contains at least two distinct points $u,v$, and consider the processes $\Phi$ and $\Psi$ whose laws are specified by \eqref{e:trans_op} and \eqref{e:trans_op_rev}, respectively, with one-step kernels defined as
$$P_1(x,\cdot):=\frac{1}{2}\left(\delta_x+\delta_u\right)\quad\text{and}\quad P_n(x,\cdot):=\frac{1}{2}\left(\delta_x+\delta_v\right)\quad\text{for}\quad x\in X,\; n\geq 2.$$
Such kernels are clearly Feller and induce the Markov operators of the form: $\mu P_1=(\mu+\delta_u)/2$ and $\mu P_n=(\mu+\delta_v)/2$ for $n\geq 2$ and $\mu\in\mathcal{M}_1(X)$. Moreover, a straightforward induction gives
$$\mu P_2^m=\frac{1}{2^m}\mu+\left(\sum_{k=1}^m \frac{1}{2^k}\right)\delta_v=\frac{1}{2^m}\mu+\left(1-\frac{1}{2^m}\right)\delta_v\quad\text{for all}\quad m\in\n,\; \mu\in\mathcal{M}_1(X).$$
Consequently, for any $n\in\n$ and $\mu\in\mathcal{M}_1(X)$,
$$\mu P^{(n)}=(\mu P_1)P_2^{n-1}=\frac{1}{2^n}(\mu+\delta_u)+\left(1-\frac{1}{2^{n-1}}\right)\delta_v,$$
$$\mu \bar{P}^{(n)}=(\mu P_2^{n-1})P_1=\frac{1}{2^n}\mu+\frac{1}{2}\left(1-\frac{1}{2^{n-1}}\right)\delta_v+\frac{1}{2}\delta_u,$$
which implies that $\mu P^{(n)}\to \delta_v=:\pi$ and $\mu \bar{P}^{(n)}\to (\delta_v+\delta_u)/2=:\bar{\pi}$ setwise (and thus weakly), as $n\to\infty$. This shows that the measures $\pi$ and $\bar{\pi}$ are attracting for $\Phi$ and $\Psi$, respectively. Nevertheless, they fail to be stationary for these processes, since
$$\pi P^{(1)}=\pi P_1=(\delta_v+\delta_u)/2\neq \pi \quad \text{and}\quad \bar{\pi} \bar{P}^{(1)}=\bar{\pi}P_1=\frac{1}{4}\delta_v+\frac{3}{4}\delta_u\neq \bar{\pi}.$$
Besides, it is evident that $\delta_u$ and $\delta_v$ are the unique invariant probability measures of $P_1$ and~$\bar{P}_1$, respectively.
\end{remark}

\begin{remark}
If $\Phi$ (resp., $\Psi$) admits a~stationary distribution $\mu_*$, and $\pi\in\mathcal{M}_1(X)$ is an attracting measure for this process, then $\pi=\mu_*$. Consequently $\pi$ is then the unique stationary distribution of $\Phi$ (resp., $\Psi$).
\end{remark}

\begin{remark}\label{rem:2}
A measure $\pi\in\mathcal{M}_1(X)$ with the property that $\delta_x P^{(n)}\stackrel{w}{\to} \pi$ (resp. $\delta_x \bar{P}^{(n)}\stackrel{w}{\to} \pi$) for every $x\in X$ is attracting for $\Phi$ (resp. $\Psi$) in the manner of \eqref{def:attracting}, provided that $X$ is Polish. To see this, fix an arbitrary $\mu\in\mathcal{M}_1(X)$ and observe that, for any $f\in L_{b,1}(X)$ and~$n\in\n$,
\begin{align}
\begin{split}
\label{e:attracting_dirac}
\left|\<f,\mu P^{(n)}\>-\<f,\pi\>\right|&=\left|\<P^{(n)}f-\<f,\pi\>,\mu\>\right|\leq \int_X \left|\<P^{(n)}f,\delta_x\>-\<f,\pi\>\right|\mu(dx)\\
&=\int_X \left|\<f,\delta_xP^{(n)}-\pi\>\right|\mu(dx)\leq \int_X d_{BL}(\delta_x P^{(n)},\pi)\mu(dx).
\end{split}
\end{align}
Hence, if $\lim_{n\to\infty} d_{BL}(\delta_x P^{(n)},\pi)=0$ for each $x\in X$, then using the Lebesgue dominated convergence theorem, we can conclude that $\lim_{n\to\infty} d_{BL}(\mu P^{(n)},\pi)=0$ as well. Exactly the same argument applies to $\bar{P}^{(n)}$ in the role of $P^{(n)}$.
\end{remark}

Let us now make one more observation. In general, neither the
convergence $\mu P^{(n)} \to \pi$ nor $\mu \bar{P}^{(n)} \to \pi$ 
for some $\mu\in\mathcal{M}_{1,1}(X)$ implies that $\pi\in \mathcal{M}_{1,1}(X)$, even if each~$(\cdot)P_n$ preserves the finiteness of first measures' moments. This implication does hold, however, for instance when $\{P_n\}_{n\in\n}$ satisfies a~uniform Lyapunov-type drift condition. More precisely, we have the following:

\begin{proposition}\label{prop:m_11}
Suppose that $\mu P^{(n)}\stackrel{w}{\to} \pi$ or $\mu \bar{P}^{(n)}\stackrel{w}{\to} \pi$ for certain $\mu,\pi\in\mathcal{M}_1(X)$, and that \hbox{$\<V,\mu\><\infty$} for some continuous function $V:X\to [0,\infty)$. Further, assume that there exist constants \hbox{$a\in [0,1)$} and $b\geq 0$ such that
\begin{equation}\label{e:lap}
P_n V(x)\leq aV(x)+b\quad\text{for all}\quad x\in X,\; n\in\n.
\end{equation}
Then $\<V,\pi\><\infty$. In particular, if $V=\rho(x_0,\cdot)$ (with some $x_0\in X$), then $\pi\in\mathcal{M}_{1,1}(X)$.
\end{proposition}
\begin{proof}
We will present the proof under the assumption that
$\mu P^{(n)} \xrightarrow{w} \pi$. If instead
\hbox{$\mu \bar{P}^{(n)} \xrightarrow{w} \pi$}, the argument is entirely analogous.

First observe that, for every $n\in\n$,
\begin{equation}\label{e:lap_ind}
P^{(n)} V(x) \leq a^n V(x) + b\sum_{k=0}^{n-1} a^k\quad\text{for all}\quad x\in X.
\end{equation}
To see this, we proceed by induction. For $n=1$, inequality \eqref{e:lap_ind} follows directly from \eqref{e:lap}. Assuming that \eqref{e:lap_ind} holds for some arbitrarily fixed $n\in\n$, by the use of \eqref{e:lap} and the inductive hypothesis, we obtain
\begin{align*}
P^{(n+1)}V=P^{(n)}P_{n+1} V\leq a P^{(n)}V+b\leq a \left(a^n V + b\sum_{k=0}^{n-1} a^k \right)+b
=a^{n+1} V+b\sum_{k=0}^n a^k,
\end{align*}
as claimed. It should be noted here that in the case of $\bar{P}^{(n)}$ we get $\bar{P}^{(n+1)}V=P_{n+1}\bar{P}^{(n)}V$; thus then one first applies the inductive hypothesis and then uses \eqref{e:lap}.

Now, for every $k\in\n$, define $V_k=\min\{V,k\}$. Clearly, $V_k\in C_b(X)$ for all $k\in\n$, and $V_k(x) \uparrow V(x)$ as $k\to \infty$ for every $x\in X$. Further, from \eqref{e:lap_ind} it follows that
\begin{align*}
\<V_k,\, \mu P^{(n)}\>&=\<P^{(n)}V_k,\, \mu\>\leq \<P^{(n)}V,\, \mu\>\leq a^n\<V,\mu\>+b\sum_{k=0}^{\infty} a^k\\
&=a^n\<V,\mu\>+\frac{b}{1-a}\quad\text{for all} \quad k,n\in\n,
\end{align*}
where $\<V,\mu\><\infty$ by the assumption. This, in turn, gives
\begin{align*}
\<V_k,\pi\>&\leq \left|\<V_k,\pi\>-\<V_k,\, \mu P^{(n)}\>\right| + \<V_k,\, \mu P^{(n)}\>\\
&\leq \left|\<V_k,\pi\>-\<V_k,\, \mu P^{(n)}\>\right|+a^n\<V,\mu\>+\frac{b}{1-a}\quad\text{for all}\quad k,n\in\n.
\end{align*}
Moreover, the convergence $\mu P^{(n)}\stackrel{w}{\to} \pi$ ensures that $\lim_{n\to\infty}\<V_k,\mu P^{(n)}\>=\<V_k,\pi\>$ for every $k\in\n$. Hence, passing to the limit as $n\to \infty$ in the latter inequality, we see that
$$\<V_k,\pi\>\leq \frac{b}{1-a}\quad\text{for all}\quad k\in\n.$$
Finally, applying the Lebesgue monotone convergence theorem gives $\<V,\pi\>\leq b/(1-a)$, which completes the proof.
\end{proof}

The remainder of this paper, within the framework of models outlined in the introduction, focuses on the situation where there exists a~measure $\pi\in\mathcal{M}_1(X)$ such that, for some $q\in (0,1)$ and $\mathcal{C}: \mathcal{M}_{1,1}(X)\to \mathbb{R}_+$,
$$
d_{BL}(\mu P^{(n)}, \pi)\leq C(\mu)q^n\quad\text{for any}\quad \mu\in\mathcal{M}_{1,1}(X),\;n\in\mathbb{N}.
$$
and likewise if $P^{(n)}$ is replaced by $\bar{P}^{(n)}$. A measure $\pi$ with this property will be called \emph{geometrically attracting} in the bounded Lipchitz distance. Of course, even a~geometrically attracting probability measure does not have to be a~stationary distribution either. For instance, despite not being invariant, the measures $\pi$ and $\bar{\pi}$ in Remark \ref{rem:1} are geometrically attracting with~$C(\mu)\equiv 4$.

\begin{remark}\label{rem:3}
Every Borel probability measure on $X$ that is geometrically attracting in $d_{BL}$ is also attracting in the sense of \eqref{def:attracting}, provided that $X$ is Polish. This follows directly from Remark \ref{rem:2}, since Dirac measures are members of~$\mathcal{M}_{1,1}(X)$.
\end{remark}

\section{Markov model given by the forward iterates}\label{sec:2}

Let $(X,\rho)$ be a~Polish metric space, and let $\{(\Theta_n, \mathcal{A}_n,\vartheta_n)\}_{n\in \n}$ be an arbitrary sequence of probability spaces. Further, for each $n\in\n$, consider a~family $\{S_{\theta}^{(n)}:\; \theta\in \Theta_n\}$ of transformations from $X$ into itself such that the map 
\begin{equation}\label{e:measurabilty}
X\times \Theta_n \ni (x,\theta) \mapsto S_{\theta}^{(n)}(x)\in X
\end{equation} 
is $\mathcal{B}(X)\otimes \mathcal{A}_n\, /\, \mathcal{B}(X)$-measurable. Given $\mu\in\mathcal{M}_1(X)$, we consider a~stochastic process \hbox{$\Phi^{\mu}:=\{\Phi_n^{\mu}\}_{n\in\n_0}$ } evolving on $X$ so that
\begin{equation}
\label{e:model_def}
\Phi_0^{\mu}\sim\mu,\quad \Phi_n^{\mu}=S_{\eta_n}^{(n)}(\Phi_{n-1}^{\mu})\;\;\text{a.s.}\quad\text{for every}\quad n\in\n,
\end{equation}
where $\{\eta_n\}_{n\in\n}$ is a~sequence of mutually independent random variables, independent of $\Phi_0^{\mu}$, such that $\eta_n$ takes values in $\Theta_n$ and has distribution $\vartheta_n$ for every $n\in\n$. Equivalently, the variables $\Phi_n^{\mu}$, $n\in\n$, can be defined as
\begin{equation}
\label{e:model_def_exp}
\Phi^{\mu}_n=\left(S_{\eta_n}^{(n)}\circ \ldots \circ S_{\eta_1}^{(1)}\right)\left(\Phi^{\mu}_0\right).
\end{equation}
For brevity, given any $n\in\mathbb{N}$ and $\theta_1\in \Theta_1,\ldots,\theta_n\in \Theta_n$, we further use the notation
$$\mathcal{S}_{\theta_1,\ldots,\theta_n}:=S_{\theta_n}^{(n)}\circ \ldots \circ S_{\theta_1}^{(1)},$$
with the convention that $\mathcal{S}_{\theta_1,\dots,\theta_0}$ represents the identity map on $X$.

Obviously, $\Phi^{\mu}$ enjoys the Markov property, and for any $x\in X$ and $A\in\mathcal{B}(X)$,
$$\pr\left(\Phi_n^{\mu}\in A\,|\,\Phi^{\mu}_{n-1}=x\right)=\pr\left(S_{\eta_n}^{(n)}(x)\in A\right)=\int_{\Theta_n} \mathbbm{1}_A\left(S_{\theta}^{(n)}(x)\right)\vartheta_n(d\theta)$$
Thus, $\Phi^{\mu}$ is a~(generally time-inhomogeneous) Markov chain with one-step transition kernels $P_n$, $n\in\n$, given by
\begin{equation}
\label{e:kernel_def}
P_n(x,A):=\int_{\Theta_n} \mathbbm{1}_A\left(S_{\theta}^{(n)}(x)\right)\vartheta_n(d\theta)\quad\text{for}\quad x\in X,\; A\in\mathcal{B}(X).
\end{equation}
Consequently, the associated Markov operators take the form
\begin{align}
\begin{split}
\label{e:op_def}
\mu P_n(A)&= \int_X \int_{\Theta_n} \mathbbm{1}_A\left(S_{\theta}^{(n)}(x)\right)\vartheta_n(d\theta)\,\mu(dx)\\
&=\int_{\Theta_n} \mu\left(\left( S_{\theta}^{(n)}\right)^{-1}(A) \right)\vartheta_n(d\theta)
\quad \text{for}\quad \mu\in\mathcal{M}_1(X),\;A\in\mathcal{B}(X),
\end{split}
\end{align}
and their dual action can be written as
\begin{equation}
\label{e:op_dual_def}
P_nf(x) = \int_{\Theta_n} f\left(S_{\theta}^{(n)}(x)\right)\vartheta_n(d\theta) \quad\text{for}\quad f\in B_b(X),\; x\in X.
\end{equation}

Throughout the remainder of this section, the stochastic kernels $P^{(n)}$, $n\in\mathbb{N}$, defined in~\eqref{e:composition}, are understood with $P_n$ given by \eqref{e:kernel_def}. Accordingly, $\{P^{(n)}\}_{n\in\n}$ describes the evolution of the distributions of the Markov chain specified by \eqref{e:model_def}.

\begin{remark}\label{rem:Feller}
The Markov operators $P_n$, $n\in\n$, given by \eqref{e:op_def}, are Feller whenever the transformations $S_{\theta}^{(n)}$, $\theta\in \Theta_n$, $n\in\n$, are continuous.
\end{remark}

Finally, for technical convenience, we shall assume that the chains $\Phi^x:=\Phi^{\delta_x}$, $x\in X$, are defined on the common probability space $(\Omega,\mathcal{F},\pr)$ of the form
\begin{equation}\label{df:ps}
\Omega=\prod_{n=1}^{\infty}\Theta_n,\quad \mathcal{F}=\bigotimes_{n=1}^{\infty}\mathcal{A}_n,\quad \pr=\bigotimes_{n=1}^{\infty}\vartheta_n
\end{equation}
via
$$
\Phi^{x}_0(\omega):=x,\quad \Phi^{x}_n(\omega):=S_{\eta_1(\omega),\ldots,\eta_n(\omega)}(x) \quad\text{for}\quad \omega\in\Omega,\;x\in X,\;n\in\mathbb{N},
$$
where $\{\eta_n\}_{n\in\n}$ is given by
\begin{equation}\label{df:eta}
\eta_n(\omega):=\theta_n\quad\text{for}\quad \omega=(\theta_1,\theta_2,\ldots)\in \Omega,\;n\in\mathbb{N}.
\end{equation}

We are now in a~position to formulate the main result of this section, which provides a~set of conditions guaranteeing the existence of a~geometrically attracting (in~$d_{BL}$) distribution for the Markov chain $\{\Phi_n^{\,\bullet}\}_{n\in\n_0}$, defined by \eqref{e:model_def}.

\begin{theorem}\label{thm:main}
Suppose that there exist measurable functions $L_n:\Theta_n\to (0,\infty)$, $n\in\n$, and \hbox{$M_n:\prod_{k=1}^n\Theta_k\to [0,\infty)$}, $n\geq n_0$, with some $n_0\in\n$, as well a~point $x_0\in X$, such~that the following conditions hold:

\begin{enumerate}[label=\textnormal{(A\arabic*)}, leftmargin=*]
\item\label{cnd:A1} For each $n\in\n$ and $\vartheta_n$-a.e. $\theta\in \Theta_n$, the transformation $S_{\theta}^{(n)}$ is Lipschitz continuous with the constant $L_n(\theta)$, i.e.,
$$\rho\left(S_{\theta}^{(n)}(x), S_{\theta}^{(n)}(y) \right)\leq L_n(\theta)\rho(x,y)\quad\text{for any}\quad x,y\in X;$$

\item\label{cnd:A2} For every $E\in\mathcal{B}((0,\infty))$, $\vartheta_n(L_n^{-1}(E))$ does not depend on $n$, and
\begin{equation}\label{eq:L_conditions}
R:=\int_{\Theta_1} L_1(\theta)\, \vartheta_1(d\theta)<\infty, \quad \int_{\Theta_1} \ln L_1(\theta)\, \vartheta_1(d\theta)<0;
\end{equation}

\item\label{cnd:A3} For every $n\geq n_0$ and $\bigotimes_{k=1}^n \vartheta_k$-a.e. $(\theta_1,\ldots,\theta_n)\in\prod_{k=1}^n\Theta_k$, the transformations $S_{\theta_1}^{(1)}, \ldots, S_{\theta_n}^{(n)}$ are invertible, and
$$
\rho\left(x_0,\, \left(\mathcal{S}_{\theta_1,\dots,\theta_{n-1}}^{-1}\circ S_{\theta_n}^{(n)} \circ \mathcal{S}_{\theta_1,\dots,\theta_{n-1}}\right) (x_0)\right)\leq M_n(\theta_1,\ldots,\theta_n);
$$

\item\label{cnd:A4}
$\displaystyle\delta:=\sup_{n\geq n_0} \int_{\Theta_1}\ldots \int_{\Theta_n} M_n(\theta_1,\ldots,\theta_n)\;\vartheta_n(d\theta_n)\ldots \vartheta_1(d\theta_1)<\infty.$

\end{enumerate}
Then there exists a~measure $\pi\in\mathcal{M}_1(X)$ that is geometrically attracting in $d_{BL}$ for the Markov chain $\Phi^{\bullet}$, specified by \eqref{e:model_def}. More precisely, there exist $q\in (0,1)$ and $C\in\mathbb{R}_+$ such that
\begin{equation}
\label{e:main}
d_{BL}\left(\mu P^{(n)},\pi\right)\leq Cq^n(1+\<V,\mu\>)\quad\text{for any}\quad \mu\in\mathcal{M}_{1,1}(X),\;n\in\mathbb{N},
\end{equation}
with $V(x):=\rho(x,x_0)$ for $x\in X$, and, in particular, $\mu P^{(n)}\stackrel{w}{\to}\pi$ for every $\mu\in\mathcal{M}_1(X)$.
\end{theorem}

\begin{remark}\label{rem:fin_int}
As for condition \ref{cnd:A2}, it can be shown that, if $R<\infty$, then the integral $\int_{\Theta_1} \ln L_1(\theta)\, \vartheta_1(d\theta)$ exists and lies in $[-\infty, \infty)$. This follows from Lemma~\ref{lem:diaconis} and Remark~\ref{rem:czeb}, given in Section \ref{sec:5}. Moreover, the Jensen inequality implies that \eqref{eq:L_conditions} holds whenever $R<1$.
\end{remark}

\begin{remark}\label{rem:X_bounded}
Note that conditions \ref{cnd:A3} and \ref{cnd:A4} are trivially satisfied (with $M_n\equiv 0$) whenever $S_{\theta}^{(n)}$ is invertible for $\vartheta_n$-a.e. $\theta\in\Theta_n$ and every $n\in\n$, and \hbox{$S_{\theta}^{(n)}(x_0)=x_0$} for $\vartheta_n$-a.e. $\theta\in\Theta_n$ and each $n\geq n_0$. In this case, $\delta_{x_0}$ is the unique stationary distribution of $\Phi^{\bullet}$, and Theorem~\ref{thm:main} is then valid with~$\pi=\delta_{x_0}$.
\end{remark}

\begin{remark}\label{rem:pi_is_inv}
Suppose that $(\Theta_n, \mathcal{A}_n, \vartheta_n)=(\Theta_1, \mathcal{A}_1, \vartheta_1)$, $S_{\theta}^{(n)}=S_{\theta}^{(1)}$ for all $n\in\n$ and~\hbox{$\theta\in \Theta_1$}, and that $S_{\theta}^{(1)}$ is continuous for $\vartheta_1$-a.e. $\theta\in \Theta_1$ (as is the case, e.g., when \ref{cnd:A1} holds). Then, if the conclusion of Theorem \ref{thm:main} is fulfilled with a~measure~$\pi\in\mathcal{M}_1(X)$, this measure must be the unique stationary distribution of the chain $\Phi^{\bullet}$. Indeed, within the given setting, $\Phi^{\bullet}$ is a~time-homogeneous Markov chain with transition law~$P_1$, which is Feller (see Remark~\ref{rem:Feller}). The fact that $\mu P_1^n=\mu P^{(n)}\stackrel{w}{\to}\pi$ for all $\mu\in\mathcal{M}_1(X)$ then easily implies the claim.
\end{remark}

The proof of Theorem \ref{thm:main}, along with all necessary auxiliary results, will be provided in Section \ref{sec:5}. As a~complement to this theorem, we present below a~result based on Proposition~\ref{prop:m_11}, ensuring that the measure $\pi$ has the finite first moment.

\begin{proposition}\label{prop:finite_moment}
Suppose that conditions \ref{cnd:A1} and \ref{cnd:A2} hold with $R<1$, and that 
\begin{equation}\label{e:war_moment} 
b:=\sup_{n\in\n}\int_{\Theta_n}\rho\left(S_{\theta}^{(n)}(x_0),\,x_0\right)\vartheta_n(d\theta)<\infty.
\end{equation}
Then the convergence $\mu P^{(n)}\stackrel{w}{\to}\pi$ for certain $\mu\in\mathcal{M}_{1,1}(X)$ and $\pi\in\mathcal{M}_1(X)$ implies that~\hbox{$\pi\in \mathcal{M}_{1,1}(X)$}. In particular, if hypotheses \ref{cnd:A3} and \ref{cnd:A4} also hold, then the measure $\pi$ whose existence is asserted in Theorem~\ref{thm:main} belongs to~$\mathcal{M}_{1,1}(X)$.
\end{proposition}
\begin{proof}
In view of Proposition \ref{prop:m_11}, it suffices to verify that \eqref{e:lap} holds with $V=\rho(\cdot,x_0)$, $a=R$, and $b$ defined as above. This, however, can be simply deduced by referring to \eqref{e:op_dual_def} and the imposed assumptions. Indeed, for any $n\in\mathbb{N}$ and $x\in X$, we get
\begin{align*}
P_n V(x)&= \int_{\Theta_n} \rho\left(S_{\theta}^{(n)}(x), x_0\right)\,\vartheta_n(d\theta)\\
&\leq \int_{\Theta_n} \rho\left(S_{\theta}^{(n)}(x), S_{\theta}^{(n)}(x_0)\right)\,\vartheta_n(d\theta)+\int_{\Theta_n} \rho\left( S_{\theta}^{(n)}(x_0),\, x_0\right)\,\vartheta_n(d\theta)\\
&\leq  \left(\int_{\Theta_n} L_n(\theta)\,\vartheta_n(d\theta)\right)\rho(x,x_0)+ \sup_{n\in\n}\int_{\Theta_n}\rho\left( S_{\theta}^{(n)}(x_0),\, x_0\right)\vartheta_n(d\theta)=RV(x)+b,
\end{align*}
which ends the proof.
\end{proof}

It is also worth noting the following straightforward consequence of Theorem \ref{thm:main}:

\begin{corollary}\label{cor:LM-const}
Suppose that there exists $L\in  (0,1)$ such that
$$\rho\left(S_{\theta}^{(n)}(x), S_{\theta}^{(n)}(y) \right)\leq L\rho(x,y)\quad \text{for}\quad x,y\in X,\; \vartheta_n\text{-a.e.}\;\theta\in\Theta_n,\; n\in\n,$$
and that there exist $x_0\in X$, $\delta\in \mathbb{R}_+$, and $n_0\in\n$ such that, for every $n\geq n_0$ and \hbox{$\bigotimes_{k=1}^n\vartheta_k$-a.e.} $(\theta_1,\ldots,\theta_n)\in\prod_{k=1}^n \Theta_k$, the transformations $S_{\theta_1}^{(1)},\ldots,S_{\theta_n}^{(n)}$ are invertible, and
$$\rho\left(x_0,\, \left(\mathcal{S}_{\theta_1,\dots,\theta_{n-1}}^{-1}\circ S_{\theta_n}^{(n)} \circ \mathcal{S}_{\theta_1,\dots,\theta_{n-1}}\right) (x_0)\right)\leq \delta.$$
Then the conclusion of Theorem \ref{thm:main} is valid.
\end{corollary}

It should be highlighted that, in general, Theorem \ref{thm:main} does not guarantee that the measure $\pi$ is stationary for the chain $\Phi ^{\bullet }$. This is demonstrated in the example below.
\begin{example}\label{rem:2.5}
Consider $(X,\rho):=(\mathbb{R},|\cdot|)$, and let \hbox{$(\Theta_n,\mathcal{A}_n,\vartheta_n):=(\Theta, \mathcal{A},\vartheta)$,} $n\in\n$, where $(\Theta, \mathcal{A},\vartheta)$ is an arbitrary probability space. Furthermore, define
$$S_{\theta}^{(1)}(x)=\frac{x}{2}-\frac{1}{2}\quad\text{and}\quad S_{\theta}^{(n)}(x)=\frac{x}{2}\quad\text{for}\quad n\geq 2,\;\theta\in\Theta.$$
Clearly, these transformations are invertible contractions, and
$$\left(\mathcal{S}_{\theta_1,\dots,\theta_{n-1}}^{-1}\circ S_{\theta_n}^{(n)} \circ \mathcal{S}_{\theta_1,\dots,\theta_{n-1}}\right) (1)=1 \quad\text{for any}\quad n\geq 2,\;\theta\in \Theta,$$
which ensures that the assumptions of Corollary \ref{cor:LM-const} (and thus Theorem \ref{thm:main}) are fulfilled (with $L=1/2$, $x_0=1$, $\delta=0$, and $n_0=2$). Moreover, since $S_{\theta}^{(1)}(1)=0$ and $S_{\theta}^{(2)}(0)=0$ for all $\theta\in\Theta$, it follows that $\delta_1 P_1=\delta_0$ and $\delta_0 P_2 =\delta_{0}$. Hence $$\delta_1 P^{(n)}=\delta_1P_1P_2^{n-1}=\delta_0 P_2^{n-1}=\delta_0\quad\text{for}\quad n\in\n,$$ which means that the conclusion of Theorem \ref{thm:main} must be valid with $\pi=\delta_0$. This measure, however, is not stationary for $\Phi^{\bullet}$ since $\delta_0 P_1=\delta_{-1/2}$.
\end{example}


\section{Stochastic model given by the backward iterates}\label{sec:3}
As before, consider a~Polish metric space $(X,\rho)$, an arbitrary sequence $\{(\Theta_n, \mathcal{A}_n,\vartheta_n)\}_{n\in \n}$ of probability spaces, and families $\{S_{\theta}^{(n)}:\; \theta\in \Theta_n\}$, $n\in\n$, of transformations from $X$ into itself such that the maps \eqref{e:measurabilty} are product measurable. Given any $\mu\in\mathcal{M}_1(X)$, we are now concerned with the $X$-valued random process \hbox{$\Psi^{\mu}:=\{\Psi_n^{\mu}\}_{n\in\n_0}$}, where
\begin{equation}
\label{e:model_def_rev}
\Psi^{\mu}_0\sim \mu,\quad \Psi^{\mu}_n:=\left(S_{\eta_1}^{(1)}\circ \ldots \circ S_{\eta_n}^{(n)}\right)\left(\Psi_0^{\mu}\right) \quad\text{for}\quad \;x\in X,\;n\in\mathbb{N},
\end{equation}
and $\{\eta_n\}_{n\in\n}$ is a~sequence of mutually independent random variables, independent of $\Psi^{\mu}_0$, such that $\eta_n$ takes values in $\Theta_n$ and $\eta_n\sim\vartheta_n$ for every $n\in\n$. Additionally, following our earlier notational convention, we set
$$
\bmathcal{S}_{\theta_1,\ldots,\theta_n}:=S_{\theta_1}^{(1)}\circ \ldots \circ S_{\theta_n}^{(n)}\quad\text{for}\quad n\in\n,\; \theta_1\in\Theta_1,\ldots,\theta_n\in\Theta_n.
$$

Moreover, similarly to Section~\ref{sec:2}, we assume that $\Psi^x:=\Psi^{\delta_x}$, $x\in X$, are constructed on the probability space~$(\Omega,\mathcal{A},\pr)$ of the form~\eqref{df:ps}, so that
$$\Psi_0^x(\omega)= x, \quad \Psi_n^x(\omega)=\bar{S}_{\eta_1(\omega),\ldots,\eta_n(\omega)}(x)\quad\text{for} \quad \omega\in\Omega,\;x\in X,\;n\in\n,$$
and the sequence $\{\eta_n\}_{n\in\n}$ is defined by~\eqref{df:eta}.

Of course, in general, $\{\Psi_n^{\mu}\}_{n\in\n_0}$ does not exhibit the Markov property. However, defining the stochastic kernels $\bar{P}^{(n)}$ by \eqref{e:composition_rev} with $P_n$ given by \eqref{e:kernel_def}, it can be shown that, for every $n\in\n$, the measure $\mu \bar{P}^{(n)}$ coincides with the distribution of $\Psi_n^{\mu}$, i.e., \eqref{e:trans_op_rev} holds. This follows immediate from the fact that
\begin{equation}\label{e:dist_psi}
\pr\left(\Psi_n^x\in A\right)=\bar{P}^{(n)}(x,A)\quad\text{for any}\quad x\in X,\; A\in\mathcal{B}(X).
\end{equation}
To verify the latter, fix $x\in X$ and $A\in\mathcal{B}(X)$. For $n=1$ we simply have
$$\pr\left(\Psi_1^x\in A\right)=\pr\left(S_{\eta_1}^{(1)}(x)\in A\right)=\int_{\Theta_1} \mathbbm{1}_A\left(S_{\theta}^{(1)}(x)\right)\vartheta_1(d\theta)=P_1(x,A)=\bar{P}^{(1)}(x,A).$$
Now, suppose that \eqref{e:dist_psi} holds for some arbitrarily fixed $n\in\n$ and note that
\begin{align*}
\bar{P}^{(n+1)}(x,A)&=\bar{P}^{(n+1)}\mathbbm{1}_A(x)=P_{n+1}\bar{P}^{(n)}\mathbbm{1}_A(x)
=\int_{\Theta_{n+1}} \bar{P}^{(n)}\mathbbm{1}_A\left(S_{\theta}^{(n+1)}(x)\right)\vartheta_{n+1}(d\theta)\\
&= \int_{\Theta_{n+1}}  \bar{P}^{(n)}\left(S_{\theta}^{(n+1)}(x),A\right)\vartheta_{n+1}(d\theta),
\end{align*}
where the third equality is due to \eqref{e:op_dual_def}. Then, using Fubini’s theorem, the inductive hypothesis, and the above identity, we obtain
\begin{align*}
\pr(&\Psi^x_{n+1}\in A)=\int_{\Theta_1\times\ldots\times\Theta_{n+1}} \mathbbm{1}_A\left(\bmathcal{S}_{\theta_1,\ldots,\theta_{n+1}}(x)\right)(\vartheta_1\otimes\ldots\otimes\vartheta_{n+1})(d\theta_1\times\ldots\times d\theta_{n+1})\\
&=\int_{\Theta_{n+1}} \int_{\Theta_1\times\ldots\times\Theta_n} \mathbbm{1}_A  \left(\bmathcal{S}_{\theta_1,\ldots,\theta_n}\left( S_{\theta_{n+1}}^{(n+1)}(x)\right)\right)(\vartheta_1\otimes\ldots\otimes\vartheta_n)(d\theta_1\times\ldots\times d\theta_n)\,\vartheta_{n+1}(d\theta_{n+1})\\
&=\int_{\Theta_{n+1}} \pr\left(\Psi_n^{S_{\theta}^{(n+1)}(x)}\in A\right)  \vartheta_{n+1}(d\theta)=\int_{\Theta_{n+1}}\bar{P}^{(n)}
\left(S_{\theta}^{(n+1)}(x),\,A\right)  \,\vartheta_{n+1}(d\theta)\\
&=\bar{P}^{(n+1)}(x,A),
\end{align*}
which implies the desired claim by the induction argument.

The main result of this section, stated below, constitutes a~counterpart to Theorem~\ref{thm:main} for the process $\{\Psi_n^{\,\bullet}\}_{n\in\n_0}$. Since this process relies on forward iterates, an analogous conclusion can be established under much less restrictive conditions; most importantly, the invertibility of $S_{\theta}^{(n)}$ is no longer required.

\begin{theorem}\label{thm:main2}
Suppose that hypotheses \ref{cnd:A1} and \ref{cnd:A2} of Theorem \ref{thm:main} are satisfied with some measurable functions $L_n:\Theta_n\to (0,\infty)$, $n\in\n$, and that there exist $x_0\in X$ and $n_0\in\n$ such that
\begin{enumerate}[label=\textnormal{(A\arabic*')}, leftmargin=*, start=3]
\item\label{cnd:B3} $\displaystyle \Delta:=\sup_{n\geq n_0} \int_{\Theta_n}\rho\left(x_0,\,S_{\theta}^{(n)}(x_0)\right)\vartheta_n(d\theta)<\infty.$
\end{enumerate}
Then there exists a~measure $\bar{\pi}\in\mathcal{M}_1(X)$ that is geometrically attracting in $d_{BL}$ for the process~$\Psi^{\bullet}$, specified by \eqref{e:model_def_rev}. More specifically, there exist $\bar{q}\in (0,1)$ and $\bar{C}\in\mathbb{R}_+$ such that
\begin{equation}
\label{e:main2}
d_{BL}\left(\mu \bar{P}^{(n)},\bar{\pi}\right)\leq \bar{C}\bar{q\,}^n(1+\<V,\mu\>)\quad\text{for any}\quad \mu\in\mathcal{M}_{1,1}(X),\;n\in\mathbb{N},
\end{equation}
where $V(x):=\rho(x,x_0)$ for $x\in X$, and, in particular, $\mu \bar{P}^{(n)}\stackrel{w}{\to}\pi$ for every $\mu\in\mathcal{M}_1(X)$.
\end{theorem}

\begin{remark}\label{rem:X_bounded2}
It is clear that \ref{cnd:B3} is trivially satisfied whenever~$X$ is bounded or there exists $n_0\in\n$ such that $S_{\theta}^{(n)}(x_0)=x_0$ for $\vartheta_n$-a.e. $\theta\in\Theta_n$ and every $n\geq n_0$. In the latter case, Theorem~\ref{thm:main2} is obviously valid with $\bar{\pi}=\delta_{x_0}$, which is the unique stationary distribution of the process~$\Psi^{\bullet}$.
\end{remark}

The proof of Theorem \ref{thm:main2} is given in Section \ref{sec:5}. For now, let us state an analogue of Proposition \ref{prop:finite_moment} (which can be established in the same way), along with some straightforward consequences of the above theorem.

\begin{proposition}\label{cor:finite_moment}
Suppose that the hypotheses of Theorem \ref{thm:main2} hold with $R<1$ and $n_0=1$ (where~$R$ is defined in \ref{cnd:A2}). Then the measure $\bar{\pi}$ satisfying its conclusion belongs to~$\mathcal{M}_{1,1}(X)$.
\end{proposition}

\begin{corollary}\label{cor:LM-const2}
Suppose that there exists $L\in (0,1)$ such that
$$\rho\left(S_{\theta}^{(n)}(x), S_{\theta}^{(n)}(y) \right)\leq L\rho(x,y)\quad \text{for any}\quad x,y\in X,\;\vartheta_n\text{-a.e.}\;\; \theta\in\Theta_n,\;n\in\n,$$
and that there exist $x_0\in X$, $\Delta\in\mathbb{R}_+$, and $n_0\in\n$ such that 
$$\rho\left(x_0, S_{\theta}^{(n)}(x_0)\right)\leq \Delta\quad \text{for}\quad \vartheta_n\text{-a.e.}\;\;\theta\in \Theta_n\quad\text{and}\quad n\geq n_0.$$
Then the conclusion of Theorem \ref{thm:main2} is valid and $\bar{\pi}\in\mathcal{M}_{1,1}(X)$.
\end{corollary}

\begin{corollary}
Suppose that $(\Theta_n, \mathcal{A}_n, \vartheta_n)=(\Theta_1, \mathcal{A}_1, \vartheta_1)$ and $S_{\theta}^{(n)}=S_{\theta}^{(1)}$ for all $n\in\n$ and~\hbox{$\theta\in \Theta_1$}. Further, assume that hypotheses \ref{cnd:A1}, \ref{cnd:A2}, and \ref{cnd:B3} hold (with $L_n=L_1$ for $n\in \n$). Then, the Markov chain $\Phi^{\bullet}$, specified by \eqref{e:model_def}, has a~unique stationary distribution, which is exponentially attracting in~$d_{BL}$.
\end{corollary}
\begin{proof}
Within the given setting, $\Phi_n^x$ and $\Psi_n^x$ share the same distribution for any $n\in\n$ and~$x\in X$, which, in turn, implies that $\mu \bar{P}^{(n)}=\mu P^{(n)}$ for all $n\in\n$ and~\hbox{$\mu\in\mathcal{M}_1(X)$}. In this case, the conclusions of  Theorem~\ref{thm:main} and Theorem \ref{thm:main2} are therefore equivalent. Consequently, the latter implies the existence of an exponentially attracting distribution $\pi$ for the chain $\Phi$. Finally,  Remark~\ref{rem:pi_is_inv} guarantees that $\pi$ is the unique invariant distribution of this chain.
\end{proof}

\section{Application to systems of affine transformations}\label{sec:4}
In this section, we will focus on the case where $S_{\theta}^{(n)}$ are affine transformations on a~Banach space. Unless stated otherwise, we assume throughout that  $\{(\Theta_n, \mathcal{A}_n,\vartheta_n)\}_{n\in \n}$ is an arbitrary sequence of probability spaces, and that $(X,\norma{\cdot})$ is a~separable Banach space over $\mathbb{K}\in\{\mathbb{R},\mathbb{C}\}$. Additionally, by $\onorma{\cdot}$ we denote the operator norm in the space of continuous linear operators from $X$ into itself. 

Adapting Theorem \ref{thm:main} (in conjunction with Proposition \ref{prop:finite_moment}) to the setting described above gives the following result:

\begin{proposition}\label{prop:affine_general}
Let \hbox{$\mathbb{A}_n:\Theta_n\to X^X$}, \hbox{$b_n:\Theta_n\to X$}, and $L_n:\Theta_n\to (0,\infty)$, $n\in\n$, be such that, for every $n\in\n$, the functions $\A{n}{\theta}:X\to X$, $\theta\in\Theta_n$, are invertible continuous linear operators, and the mappings
\begin{gather*}
b_n,\quad L_n,\quad X\times \Theta_n \ni (x,\theta) \mapsto \A{n}{\theta}(x)\in X,\\
\Theta_n \ni \theta\mapsto \onorma{\A{n}{\theta}}\in\mathbb{R},\quad\text{and}\quad \Theta_n \ni \theta\mapsto \onorma{\A{n}{\theta}^{-1}}\in\mathbb{R}
\end{gather*}
are measurable. Furthermore, suppose that the following conditions hold:

\begin{enumerate}[label=\textnormal{(B\arabic*)}, leftmargin=*]
\item\label{cnd:E1} $\onorma{\A{n}{\theta}}\leq L_n(\theta)$ for any  $\theta\in\Theta_n$ and $n\in\n$;
\item\label{cnd:E2} $\vartheta_n(L_n^{-1}(d\theta))$ does not depend on $n$, and $L_1$ satisfies \eqref{eq:L_conditions};
\item\label{cnd:E3} For every $n\in\n$, the integrals
\begin{gather*}
\displaystyle
\alpha_n:=\int_{\Theta_n}\onorma{\A{n}{\theta}^{-1}}\vartheta_n(d\theta),
\quad 
\beta_n:=\int_{\Theta_n}\norma{b_n (\theta)}\onorma{\A{n}{\theta}^{-1}}\vartheta_n(d\theta),\\[0.2cm] 
\quad 
\delta_n:=\int_{\Theta_n}\onorma{\A{n}{\theta}}\onorma{\A{n}{\theta}^{-1}} \, \vartheta_n(d\theta),
\quad 
\gamma_n:=\int_{\Theta_n} \norma{b_n(\theta)}\vartheta_n(d\theta),
\end{gather*}
are finite, and
$$\prod_{n=1}^{\infty} \delta_n<\infty,\quad \sum_{n=1}^{\infty} \frac{\beta_n}{\delta_n} \prod_{j=1}^{n-1}\frac{\alpha_j}{\delta_j}<\infty,\quad \text{and}\quad \sup_{n\in\n}\gamma_n\prod_{j=1}^{n-1}\alpha_j<\infty.$$


\end{enumerate}
Then the conclusion of Theorem \ref{thm:main} is valid for the Markov chain $\Phi^{\bullet}$ defined in \eqref{e:model_def} via the affine transformations of the form
\begin{equation}\label{def:iso}
S_{\theta}^{(n)}(x):=\A{n}{\theta}(x)+b_n(\theta)\quad\text{for}\quad x\in X,\; n\in\n,\; \theta\in\Theta_n.
\end{equation}
If additionally $R<1$ (with $R$ defined in \eqref{eq:L_conditions}), then the attracting measure $\pi$ satisfying \eqref{e:main} belongs to $\mathcal{M}_{1,1}(X)$.
\end{proposition}

\begin{proof}
To establish the main statement, it is enough to verify hypotheses \ref{cnd:A1}-\ref{cnd:A4} of Theorem \ref{thm:main}. 

Clearly, due to condition \ref{cnd:E1}, for every $n\in\n$ and  $\theta\in\Theta_n$, we have
$$\norma{S_{\theta}^{(n)}(x)-S_{\theta}^{(n)}(y)}=\norma{\A{n}{\theta}(x-y)}\leq \onorma{\A{n}{\theta}}\norma{x-y}\leq L_n(\theta)\norma{x-y}\;\;\text{for}\;\; x,y\in X.$$
Hence, also taking into account \ref{cnd:E2}, we see that hypotheses \ref{cnd:A1} and \ref{cnd:A2} are satisfied. 

Now, let $n\geq 2$, fix  $\theta_1\in \Theta_1,\ldots,\theta_n\in\Theta_n$,  and put $\sigma_k:=(\theta_1,\ldots,\theta_k)$ for $k\in\{1,\ldots,n\}$. It~is easy to check that
$$\mathcal{S}_{\sigma_{n-1}}(x)=\mathbf{A}_{\sigma_{n-1}}(x)+\mathbf{b}_{\sigma_{n-1}},\quad \mathcal{S}_{\sigma_{n-1}}^{-1}(x)=\mathbf{A}_{\sigma_{n-1}}^{-1}(x-\mathbf{b}_{\sigma_{n-1}})\quad\text{for}\quad x\in X,
$$
where
$$\mathbf{A}_{\sigma_{n-1}}=\A{n-1}{\theta_{n-1}}\circ\ldots\circ \A{1}{\theta_1},$$
$$
\mathbf{b}_{\sigma_{n-1}}=b_{n-1}(\theta_{n-1})+\sum_{k=1}^{n-2} \left(\A{n-1}{\theta_{n-1}}\circ\ldots\circ \A{k+1}{\theta_{k+1}}\right)(b_k(\theta_k)). $$
Thus, defining $\mathcal{H}_{\sigma_{n}}:=\mathcal{S}_{\sigma_{n-1}}^{-1}\circ S_{\theta_n}^{(n)}\circ \mathcal{S}_{\sigma_{n-1}}$, we get
\begin{align*}
\mathcal{H}_{\sigma_{n}}(0)&
=\mathbf{A}_{\sigma_{n-1}}^{-1}\left(\A{n}{\theta_n}\left(\mathbf{A}_{\sigma_{n-1}}(0)+\mathbf{b}_{\sigma_{n-1}} \right)+b_n(\theta_n)-\mathbf{b}_{\sigma_{n-1}} \right)\\
&=\mathbf{A}_{\sigma_{n-1}}^{-1}\left(\left(\A{n}{\theta_n}-\operatorname{id}_X\right)(\mathbf{b}_{\sigma_{n-1}})+b_n(\theta_n)\right).
\end{align*}
Bearing in mind the submultiplicity of $\onorma{\cdot}$, the continuity (boundedness) of $A_n(\theta)$, $\theta\in\Theta_n$, and hypothesis \ref{cnd:E1}, we can therefore conclude that
\begin{align*}
\norma{\mathcal{H}_{\sigma_{n}}(0)}&\leq \onorma{\mathbf{A}_{\sigma_{n-1}}^{-1}} \left(\onorma{\A{n}{\theta_n}-\operatorname{id}_X}\norma{\mathbf{b}_{\sigma_{n-1}}} +\norma{b_n(\theta_n)}\right)\\
&\leq \left(\prod_{j=1}^{n-1} \onorma{\A{j}{\theta_j}^{-1}} \right)\left(\left(\onorma{\A{n}{\theta_n}}+1\right)\sum_{k=1}^{n-1}\norma{b_k(\theta_k)}\prod_{j=k+1}^{n-1}\onorma{\A{j}{\theta_j}}+\norma{b_n(\theta_n)}\right)\\
&\leq\left(L_n(\theta_n)+1\right) \sum_{k=1}^{n-1} \boldsymbol{\Delta}_k(\theta_1,\ldots,\theta_{n-1})+\norma{b_n(\theta_n)}\prod_{j=1}^{n-1} \boldsymbol{\alpha}_j(\theta_j),
\end{align*}
with
$$\boldsymbol{\Delta}_k(\theta_1,\ldots,\theta_{n-1}):=\norma{b_k(\theta_k)} \left( \prod_{j=1}^{n-1} \onorma{\A{j}{\theta_j}^{-1}} \right)\left( \prod_{j=k+1}^{n-1}\onorma{\A{j}{\theta_j}}\right),$$
$$\boldsymbol{\alpha}_j(\theta):=\onorma{\A{j}{\theta}^{-1}}.$$
Further, note that $\boldsymbol{\Delta}_k(\theta_1,\ldots,\theta_{n-1})$ can be rewritten as
$$\boldsymbol{\Delta}_k(\theta_1,\ldots,\theta_{n-1})=\boldsymbol{\beta}_k(\theta_k)\left( \prod_{j=1}^{k-1}\boldsymbol{\alpha}_j(\theta_j)\right)\left( \prod_{j=k+1}^{n-1}\boldsymbol{\delta}_j(\theta_j)\right),$$
where 
$$\boldsymbol{\beta}_k(\theta):=\norma{b_k(\theta)}\onorma{\A{k}{\theta}^{-1}}, \quad \boldsymbol{\delta}_j(\theta):=\onorma{\A{j}{\theta}}\onorma{\A{j}{\theta}^{-1}}.$$

Consequently, we have shown that
\begin{align*}
\|\mathcal{H}_{\theta_1,\ldots,\theta_n}(0)\|&\leq M_n(\theta_1,\ldots,\theta_n) \quad\text{for any}\quad n\geq 2,\;\theta_1\in\Theta_1,\ldots,\theta_n\in\Theta_n
\end{align*}
with
$$M_n(\theta_1,\ldots,\theta_n)\hspace{-0.1cm}:=(L_n(\theta_n)+1)\sum_{k=1}^{n-1}\boldsymbol{\beta}_k(\theta_k)\left( \prod_{j=1}^{k-1}\boldsymbol{\alpha}_j(\theta_j)\right)\left( \prod_{j=k+1}^{n-1}\boldsymbol{\delta}_j(\theta_j)\right)+\norma{b_n(\theta_n)}\prod_{j=1}^{n-1}\boldsymbol{\alpha}_j(\theta_j),$$
which means that \ref{cnd:A3} holds with $x_0=0$, $n_0=2$, and $M_n$ defined as above.
Finally, from conditions \ref{cnd:E2} and \ref{cnd:E3} it follows that
\begin{align*}
\int_{\Theta_1}\ldots \int_{\Theta_n} M_n(\theta_1,\ldots,\theta_n)\;\vartheta_n(d\theta_n)\ldots &\vartheta_1(d\theta_1)= (R+1)\sum_{k=1}^{n-1} \beta_k\left(\prod_{j=1}^{k-1}\alpha_j\right)\left(\prod_{j=k+1}^{n-1}\delta_j\right)+\gamma_n\prod_{j=1}^{n-1}\alpha_j\\
&=(R+1)\left(\prod_{j=1}^{n-1} \delta_j \right)\sum_{k=1}^{n-1} \frac{\beta_k}{\delta_k}\prod_{j=1}^{k-1}\frac{\alpha_j}{\delta_j}+\gamma_n\prod_{j=1}^{n-1}\alpha_j\\
&\leq (R+1) \left(\prod_{j=1}^{\infty} \delta_j \right)\sum_{k=1}^{\infty} \frac{\beta_k}{\delta_k}\prod_{j=1}^{k-1}\frac{\alpha_j}{\delta_j}+\sup_{n\in\n}\gamma_n\prod_{j=1}^{n-1}\alpha_j<\infty,
\end{align*}
where the penultimate inequality is valid due to the fact that $\delta_j\geq 1$ for every $j\in\n$ (see~Remark~\ref{rem:alpha_delta} below). We have therefore shown that \ref{cnd:A4} is satisfied.

Finally, observe that 
$$\|S_{\theta}^{(n)}(0)\|=\norma{b_n(\theta)}\quad\text{for any}\quad n\in\n,\; \theta\in\Theta_n\quad\text{and}\quad\sup_{n\in\n} \gamma_n<\infty,\vspace{-0.3cm}$$ 
where the latter follows from the last formula in \ref{cnd:E3} and the fact that $\alpha_n>1$ for all~$n\in\n$ (see~Remark \ref{rem:alpha_delta}). This ensures that \eqref{e:war_moment} holds. Therefore, in view of Proposition \ref{prop:finite_moment}, $\pi\in\mathcal{M}_{1,1}(X)$ whenever $R<1$. The proof is now complete.
\end{proof}

\begin{remark}\label{rem:alpha_delta}
Note that, under the assumptions of Proposition \ref{prop:affine_general}, $\delta_n\geq 1$ and $\alpha_n>1$ for every $n\in\n$. Indeed, the bound on $\delta_n$ follows immediately from the fact that
$$\onorma{\A{n}{\theta}}\onorma{\A{n}{\theta}^{-1}}\geq \onorma{\A{n}{\theta}\A{n}{\theta}^{-1}}=\onorma{\operatorname{id}_X}=1.$$
As for $\alpha_n$, it suffices to use the inequality $1/t\geq 1-\ln t$ for $t>0$, together with conditions \ref{cnd:E1} and \ref{cnd:E2}; namely,
\begin{align*}
\alpha_n \geq \int_{\Theta_n}\frac{\vartheta_n(d\theta)}{\onorma{\A{n}{\theta}}}\geq \int_{\Theta_n}\frac{\vartheta_n(d\theta)}{L_n(\theta)}
\geq 1-\int_{\Theta_n} \ln L_n(\theta)\,\vartheta_n(d\theta)>1.
\end{align*}
\end{remark}

\begin{remark}
For the sequence $\{\gamma_n\}_{n\in\n}$, in Proposition~ \ref{prop:affine_general} it is in fact sufficient to assume that $\gamma_n<\infty$ for $n\geq n_0$ (with some $n_0\in \n$), and to replace the last requirement in \ref{cnd:E3} with $\sup_{n\geq n_0}\gamma_n\prod_{j=1}^{n-1}\alpha_j<\infty$. It should be kept in mind, however, that the assertion concerning the finiteness of the first moment of $\pi$ still requires $\sup_{n\in\n}\gamma_n<\infty$. 
\end{remark}

\begin{remark}\label{rem:iso}
If, for a~given $n\in\n$ and $\vartheta_n$-a.e. $\theta\in\Theta_n$, the operator $\onorma{\A{n}{\theta}}^{-1}\mathbb{A}_n(\theta)$ is an isometry, i.e., $\onorma{\mathbb{A}_n(\theta)(x)}=\onorma{\A{n}{\theta}}\norma{x}$ for all $x\in X$, then $\delta_n=1$, since the integrand in the relevant integral is identically equal to 1.

\end{remark}

\begin{remark}\label{rem:af_sup}
Under the additional assumption that there exists a~constant $c>0$ such that 
$$\onorma{\A{n}{\theta}}\leq c\quad\text{for}\quad \vartheta_n\text{-a.e.}\;\;\theta\in\Theta_n, \;n\in\n,$$
the conditions
$$\gamma_n<\infty \quad \text{for}\quad n\in\n\quad\text{and}\quad \sup_{n\in\n} \gamma_n\,\prod_{j=1}^{n-1}\alpha_j<\infty$$
result automatically from the remaining part of hypothesis~\ref{cnd:E3}. Indeed, for every~$n\in\n$,
$$\gamma_n=\int_{\Theta_n} \norma{b_n(\theta)}\onorma{\A{n}{\theta}^{-1}\A{n}{\theta}}\,\vartheta_n(d\theta)\leq
c\int_{\Theta_n} \norma{b_n (\theta)}\onorma{\A{n}{\theta}^{-1}}\vartheta_n(d\theta)= c\beta_n <\infty.$$ 
This, in turn, together with the fact $\delta_j\geq 1$ for all $j\in\n$ (cf. Remark \ref{rem:alpha_delta}), implies that
\begin{align*}
\gamma_n\,\prod_{j=1}^{n-1} \alpha_j
&\leq c\beta_n\prod_{j=1}^{n-1}\alpha_j\leq c \sum_{k=1}^n\beta_k  \left(\prod_{j=1}^{k-1} \alpha_j\right)\left(\prod_{j=k+1}^n\delta_j \right)=c\left(\prod_{j=1}^n \delta_j\right)\sum_{k=1}^n \frac{\beta_k}{\delta_k} \prod_{j=1}^{k-1}\frac{\alpha_j}{\delta_j},
\end{align*}
which finally gives
$$\sup_{n\in\n} \gamma_n\,\prod_{j=1}^{n-1} \alpha_j\leq c \left(\prod_{j=1}^{\infty} \delta_j\right)\sum_{k=1}^{\infty} \frac{\beta_k}{\delta_k} \prod_{j=1}^{k-1}\frac{\alpha_j}{\delta_j}<\infty.$$
\end{remark}

It is also worth noting that, in the case where $\mathbb{A}_n(\theta)$ are homotheties with center 0 (i.e.,~$\mathbb{A}_n(\theta)(x)=a_n(\theta)x$ for some $a_n:\Theta_n\to\mathbb{K}$), Proposition \ref{prop:affine_general} simplifies to the following:

\begin{corollary}\label{prop:affine}
Let $a_n:\Theta_n\to \mathbb{K}$, $b_n:\Theta_n\to X$, and $L_n:\Theta_n\to (0,\infty)$, $n\in \n$, be measurable functions such that \ref{cnd:E2} and the following conditions hold:
\begin{enumerate}[label=\textnormal{(B\arabic*')}, leftmargin=*]
\item\label{cnd:ex1} $0<|a_n(\theta)|\leq L_n(\theta)$ for any $\theta\in\Theta_n$ and $n\in\n$,
\setcounter{enumi}{2}
\item\label{cnd:ex3} For every $n\in\n$, the integrals
\begin{equation}\label{ex3:ints}
\displaystyle\alpha_n:=\int_{\Theta_n} \frac{1}{|a_n(\theta)|}\,\vartheta_n(d\theta), \quad \beta_n:=\int_{\Theta_n} \frac{\norma{b_n(\theta)}}{|a_n(\theta)|}\, \vartheta_n(d\theta), \quad \gamma_n:=\int_{\Theta_n} \norma{b_n(\theta)}\,\vartheta_n(d\theta)
\end{equation}
are finite, and
$$\displaystyle\sum_{n=1}^{\infty} \beta_n\, \prod_{j=1}^{n-1} \alpha_j<\infty,\quad  \sup_{n\in\n} \gamma_n\,\prod_{j=1}^{n-1}\alpha_j<\infty.$$
\end{enumerate}
Then the conclusion of Theorem \ref{thm:main} is valid for the Markov chain defined according to \eqref{e:model_def} via the transformations
\begin{equation}\label{def:af}
S_{\theta}^{(n)}(x):=a_n(\theta)x+b_n(\theta)\quad\text{for}\quad x\in X,\; n\in\n,\; \theta\in\Theta_n.
\end{equation}
If additionally $R<1$ (with $R$ defined in \emph{\ref{cnd:E2}}), then the attracting measure $\pi$ satisfying \eqref{e:main} belongs to $\mathcal{M}_{1,1}(X)$.

\end{corollary}

\begin{remark}\label{rem:af_stale}
If the functions $a_n$ and $b_n$, $n\in\n$, are constant, conditions \ref{cnd:ex1}, \ref{cnd:E2}, and~\ref{cnd:ex3} of Corollary \ref{prop:affine} reduce to the following ones:
$$0<|a_n|\leq L<1\quad\text{for every}\quad n\in\n,  \quad \text{and}\quad \sum_{n=1}^{\infty} \frac{\norma{b_n}}{|a_1\cdot\ldots\cdot a_{n}|}<\infty.$$
\end{remark}

We shall now present a~fairly wide class of affine transformations for which the hypotheses of Proposition \ref{prop:affine_general} are satisifed.

\begin{example}\label{example:1}
Let $(\Theta,\mathcal{A},\vartheta)$ be an arbitrary probability space, and, for every $n\in\n$, let $T_n:\Theta_n\to\Theta$ be an $\mathcal{A}_n/\mathcal{A}$-measurable transformation satisfying $\vartheta_n(T_n^{-1}(d\theta))=\vartheta(d\theta)$. Further, suppose that we are given functions $\widehat{b}:\Theta\to X$, $L:\Theta \to (0,\infty)$, and \hbox{$\wA(\theta):\Theta\to X^X$} such that $\wA(\theta):X\to X$, $\theta\in\Theta$, are invertible continuous linear operators, and the mappings
\begin{gather}\label{eq:cnd_L0}
\begin{split}
&\widehat{b}, \quad L,\quad X\times \Theta \ni (x,\theta) \mapsto \wA(\theta)(x)\in X, \\
&\Theta \ni \theta\mapsto \onorma{\wA(\theta)}\in\mathbb{R},\quad \text{and}\quad \Theta \ni \theta\mapsto \onorma{\wA(\theta)^{-1}}\in\mathbb{R}\quad\text{are measurable},
\end{split}
\end{gather}
and the following conditions hold:
\begin{equation}\label{eq:cnd_L1}
\onorma{\wA(\theta)}\leq L(\theta)\quad\text{for all}\quad \;\;\theta\in\Theta,
\end{equation}
\begin{equation}\label{eq:cnd_L2}
\quad \widehat{R}:=\int_{\Theta} L(\theta)\,\vartheta(d\theta)<\infty,\quad \int_{\Theta} \ln L(\theta)\,\vartheta(d\theta)<0,
\end{equation}
\begin{gather}
\label{eq:cnd_L3}
\begin{split}
\displaystyle
\widehat{\alpha}:=\int_{\Theta}\onorma{\wA(\theta)^{-1}}\vartheta(d\theta)<
\infty,
\quad 
\widehat{\beta}:=\int_{\Theta}\norma{\widehat{b}(\theta)}\onorma{\wA(\theta)^{-1}}\vartheta(d\theta)<
\infty,\\[0.2cm] 
\quad 
\widehat{\delta}:=\int_{\Theta}\onorma{\wA(\theta)}\onorma{\wA(\theta)^{-1}} \, \vartheta(d\theta)=1,
\quad 
\widehat{\gamma}:=\int_{\Theta} \norma{\widehat{b}(\theta)}\vartheta(d\theta)<
\infty.
\end{split}
\end{gather}

Moreover, let $\mathcal{C}_n,\mathcal{D}_n:\Theta_n\to \mathbb{K}$, $n\in\n$, be any measurable functions such that, for certain sequences \hbox{$\{c_n\}_{n\in\n}, \{d_n\}_{n\in\n}\subset \mathbb{R}_+$}, 
\begin{equation}\label{eq:cnd_L4}
0<c_n\leq  |\mathcal{C}_n(\theta)|\leq 1 \quad\text{and}\quad |\mathcal{D}_n(\theta)|\leq d_n\quad\text{for}\quad  \theta\in\Theta_n,\; n\in\n.
\end{equation}

Finally, consider the affine transformations of the form \eqref{def:iso} with  
\begin{equation}\label{eq:ab}
\mathbb{A}_n(\theta)=\mathcal{C}_n(\theta) \wA(T_n(\theta)) \quad\text{and}\quad b_n(\theta)=\mathcal{D}_n(\theta)  \widehat{b}(T_n(\theta)) \quad\text{for}\quad \theta\in\Theta_n,\; n\in\mathbb{N},
\end{equation}
that is,
\begin{equation}  \label{e:af_Tn}
S_{\theta}^{(n)}(x):=\mathcal{C}_n(\theta)\,\wA(T_n(\theta))(x)+\mathcal{D}_n(\theta)\,\widehat{b}(T_n(\theta))\quad \text{for}\quad x\in X,\; n\in\mathbb{N},\;\theta\in\Theta_n.
\end{equation}

It follows directly from \eqref{eq:cnd_L1}, \eqref{eq:cnd_L2} and the first part of \eqref{eq:cnd_L4} that conditions \ref{cnd:E1} and \ref{cnd:E2} hold with $L_n:=L\circ T_n$, $n\in\n$. Moreover, note that, due to \eqref{eq:cnd_L3}-\eqref{eq:cnd_L4}, the sequences defined in \ref{cnd:E3} can be estimated as follows:
$$\alpha_n \leq \frac{\widehat{\alpha}}{c_n},\quad \beta_n\leq \frac{\widehat{\beta} d_n}{c_n},\quad \delta_n=1,\quad\text{and}\quad \gamma_n\leq \widehat{\gamma}d_n\quad\text{for}\quad n\in\n.$$
Hence, for every $n\in\n$,

$$\frac{\beta_n}{\delta_n} \prod_{j=1}^{n-1}\frac{\alpha_j}{\delta_j}= \beta_n \prod_{j=1}^{n-1} \alpha_j\leq \widehat{\beta}d_n\widehat{\alpha}^{n-1}\prod_{j=1}^{n} \frac{1}{c_j}=\frac{\widehat{\beta}}{\widehat{\alpha}}d_n\widehat{\alpha}^n\prod_{j=1}^{n} \frac{1}{c_j},$$

$$\gamma_n\,\prod_{j=1}^{n-1}\alpha_j\leq \widehat{\gamma}d_n\widehat{\alpha}^{n-1}\prod_{j=1}^{n-1} \frac{1}{c_j}\leq \frac{\widehat{\gamma}}{\widehat{\alpha}}d_n\widehat{\alpha}^{n}\prod_{j=1}^{n}\frac{1}{c_j}.$$
Consequently, we see that condition \ref{cnd:E3} is fulfilled whenever
\begin{equation}\label{ex:sum}
\sum_{n=1}^{\infty}d_n\widehat{\alpha}^n\prod_{j=1}^n \frac{1}{c_j}<\infty.
\end{equation}
In view of Proposition \ref{prop:affine_general}, one may therefore conclude that the assertion of Theorem \ref{thm:main} applies to the model defined above provided that it meets conditions \eqref{eq:cnd_L0}-\eqref{eq:cnd_L4} and \eqref{ex:sum}. If additionally $\widehat{R}<1$, then the the attracting measure $\pi$ belongs to $\mathcal{M}_{1,1}(X)$.

Naturally, since $c_n\leq 1$, the convergence of the series $\sum_{n=1}^{\infty} d_n\widehat{\alpha}^n$ is necessary for condition \eqref{ex:sum} to hold. It is also worth noting that \eqref{ex:sum} is satisfied, for instance, in each of the following cases:
\begin{equation}\label{ex:sum3}
\sum_{n=1}^{\infty} \left(1-c_n\right)<\infty\quad\text{and}\quad \sum_{n=1}^{\infty} \widehat{\alpha}^n d_n <\infty,
\end{equation}
\begin{equation}\label{ex:sum1}
\underline{c}:=\inf_{n\in\n} c_n >0 \quad \text{and}\quad \sum_{n=1}^{\infty} \left(\frac{\widehat{\alpha}}{\underline{c}} \right)^n d_n <\infty.
\end{equation}
The implication \eqref{ex:sum1}\; $\Rightarrow$\;\eqref{ex:sum} is immediate, whereas \eqref{ex:sum3}\;$\Rightarrow$\;\eqref{ex:sum} follows from the well-known fact that, for any $\{u_j\}_{j\in\n}\subset [0,1)$,  \hbox{$\prod_{j=1}^{\infty} (1-u_j)<\infty$} if and only if $\sum_{j=1}^{\infty} u_j<\infty$, applied to $u_j:=1-c_j$, together with the assumption that $0<c_j\leq 1$ for every $j\in\n$.

Let us emphasize here that, although the first part of \eqref{ex:sum3} implies the first part of~\eqref{ex:sum1}, and  \eqref{ex:sum1} implies the second part of \eqref{ex:sum3}, it is easy to verify that neither  \eqref{ex:sum3} nor \eqref{ex:sum1}, taken as a whole, implies the other.
\end{example}

An illustrative class of mappings $\wA$, $\widehat{b}$, and $L$ on $(0,1)$  that satisfy conditions specified in Example \ref{example:1} (with respect to the Lebesgue measure) is presented below.

\begin{example}\label{example:2}
Consider the probability space $(\Theta,\mathcal{A},\vartheta):=((0,1),\mathcal{B}((0,1)), \lambda)$, where $\lambda$ is the Lebesgue measure restricted to $\mathcal{B}((0,1))$, and suppose that we are given an arbitrary family of surjective (and thus invertible) linear isometries $\mathbb{B}(\theta) : X \to X$, $\theta\in\Theta$, such that \hbox{$(x,\theta) \mapsto \mathbb{B}(\theta)(x)$} is measurable. For example, in the case of $X=\mathbb{R}^N$ \hbox{(resp. $X=\mathbb{C}^N$)} equipped with the Euclidean norm, such isometries take the form \hbox{$\mathbb{B}(\theta)(x)=Q(\theta)x$}, where $Q(\theta)$ are \hbox{$N\times N$} scalar orthogonal (resp. unitary) matrices such that $\theta \mapsto Q(\theta)$ is $\mathcal{A}/\mathcal{B}(M_N)$-measurable (here, $M_N$ is the space of $N\times N$ matrices endowed with the Frobenius norm).

Further, fix
$$p,q\in(-1,1), \quad \bar{p}>\max\{0,p\}-1,\quad\bar{q}>\max\{0,q\}-1,$$
and let  $F:\Theta\to \mathbb{K}$, $f:\Theta \to X$ be arbitrary measurable functions such that
$$0<m_F\leq |F(\theta)|< e^{p+q},\quad \norma{f(\theta)}\leq M_f\quad\text{for any}\quad \theta\in\Theta$$
and some  $m_F,M_f\in (0,\infty)$.

We will show that conditions \eqref{eq:cnd_L0}-\eqref{eq:cnd_L3} hold for the mappings $\wA:\Theta\to X^X$, \hbox{$\widehat{b}:\Theta\to X$}, and \hbox{$L:\Theta\to (0,\infty)$} given by
$$\wA(\theta)(x):=(1-\theta)^p\theta^qF(\theta)\,\mathbb{B}(\theta)(x),\quad \widehat{b}(\theta):=(1-\theta)^{\bar{p}}\theta^{\bar{q}}f(\theta),\quad L(\theta):=(1-\theta)^p\theta^q|F(\theta)|.$$
Clearly, since $\mathbb{B}(\theta)$, $\theta\in\Theta$, are isometries, we have
$$\onorma{\wA(\theta)}=(1-\theta)^p\theta^q|F(\theta)|=L(\theta),\quad \onorma{\wA(\theta)^{-1}}=1/L(\theta)\quad\text{for}\quad\theta\in\Theta,$$
which ensures that conditions \eqref{eq:cnd_L0} and \eqref{eq:cnd_L1} are fulfilled.

Given the upper bound $F$, we further get
\begin{align*}
\int_{\Theta} \ln L(\theta)\,d\theta&=p \int_0^1\ln(1-\theta)\,d\theta
+q \int_0^1 \ln(\theta)\,d\theta+\int_0^1 \ln |F(\theta)|\,d\theta\\
&<-p-q+p+q=0,
\end{align*}
which yields the second part of \eqref{eq:cnd_L2}. 

Next, recall that
$$\int_0^1 (1-\theta)^r \theta^s d\theta =\mathrm{B} (s+1,r+1)<\infty \quad\text{for any}\quad s,r>-1,$$
where $\mathrm{B}$ denotes the beta function. Using this, together with the bounds on $F$ and $f$, we deduce the first part of \eqref{eq:cnd_L2} and conditions \eqref{eq:cnd_L3} as follows:
$$\int_{\Theta} L(\theta)\,d\theta\leq e^{p+q} \int_0^1 (1-\theta)^p\theta^q \,d\theta=e^{p+q}\mathrm{B} (p+1,q+1),$$
$$
\widehat{\alpha}=\int_0^1 \frac{d\theta}{|L(\theta)|}\leq \frac{1}{m_F}\int_0^1 (1-\theta)^{-p}\theta^{-q}\,d\theta=\frac{\mathrm{B}(1-p,1-q)}{m_F},
$$
$$
\widehat{\beta}=\int_0^1 \frac{\norma{\widehat{b}(\theta)}}{|L(\theta)|}\,d\theta\leq \frac{M_f}{m_F}\int_0^1 (1-\theta)^{\bar{p}-p}\theta^{\bar{q}-q}\,d\theta=\frac{M_f\mathrm{B}(\bar{p}-p+1,\bar{q}-q+1)}{m_F},
$$
$$
\widehat{\delta}=\int_0^1 \frac{L(\theta)}{L(\theta)}\,d\theta=1,
$$
$$
\widehat{\gamma}=\int_0^1 \norma{\widehat{b}(\theta)}\,d\theta\leq M_f\int_0^1 (1-\theta)^{\bar{p}}\theta^{\bar{q}}\,d\theta=M_f\mathrm{B}(\bar{p}+1,\bar{q}+1).
$$
\end{example}

Let us conclude the discussion so far by presenting two concrete examples that fall within the framework of Example \ref{example:1}.

\begin{example}
Let $\{a_n\}_{n\in\n}$ be an arbitrary sequence of positive numbers, and take
$$(\Theta_n,\mathcal{A}_n,\vartheta_n)=\left((0,1/a_n),\, \mathcal{B}((0,1/a_n)),\, a_n\lambda|_{\mathcal{B}((0,1/a_n))}\right),$$ where $\lambda$ denotes the Lebesgue measure. Consider the transformations of the form 
\begin{equation}\label{ex1:S}
S_{\theta}^{(n)}(x):=(1-sa_n\theta)x+r^n\mathcal{D}(\theta)f(a_n\theta)\quad\text{for} \quad x\in X,\; \theta\in\Theta_n,\; n\in\n,
\end{equation}
where $s\in (0,1)$, $r\in\mathbb{R}$, and $\mathcal{D}:(0,\infty)\to \mathbb{R}$, $f:(0,1)\to X$ are any bounded measurable functions.

Note the the transformations defined above fall within the class discussed in Example~\ref{example:1}. Indeed, with
$$(\Theta,\mathcal{A},\vartheta)=\left((0,1),\, \mathcal{B}((0,1)),\, \lambda|_{\mathcal{B}((0,1))}\right),$$
the maps $T_n:\Theta_n\to \Theta$, $n\in\n$, defined by $T_n(\theta):=a_n\theta$ for $\theta\in\Theta_n$, are $\mathcal{B}(\Theta_n)/\mathcal{B}(\Theta)$-measurable and satisfy $\vartheta_n(T_n^{-1}(d\theta))=\vartheta(d\theta)$. Moreover, $S_{\theta}^{(n)}$ can be expressed in the form \eqref{e:af_Tn} with $\wA$, $\widehat{b}$, $\{\mathcal{C}_n\}_{n\in\n}$, and $\{\mathcal{D}_n\}_{n\in\n}$  given~by
$$\wA(\theta)=(1-s\theta)\operatorname{id}_X,\quad \widehat{b}(\theta)=f(\theta)\quad\text{for}\quad \theta\in\Theta,$$
$$\mathcal{C}_n(\theta)=1, \quad\text{and}\quad \mathcal{D}_n(\theta)=r^n\mathcal{D}(\theta)\quad\text{for}\quad \theta\in\Theta_n,\;n\in\n,$$
so that \eqref{eq:cnd_L4} holds with
$$c_n=1\quad \text{and}\quad d_n=D|r|^n \quad\text{for}\quad n\in\n, \quad\text{where}\quad D:=\sup_{\theta\in (0,\infty)} \mathcal{D}(\theta).$$
Furthermore, the mappings $\wA,\widehat{b}$, and $L:=1-s\operatorname{id}_{\Theta}$ belong to the family introduced in Example~\ref{example:2} with $p=q=\bar{p}=\bar{q}=0$, $F=L$, $\mathbb{B}(\theta)=\operatorname{id}_X$ for $\theta\in\Theta$, and the given $f$, since \hbox{$0<1-s\leq F(\theta)<1=e^{0}$} for all~$\theta\in\Theta$, and $f$ is bounded. As established previously, such functions satisfy conditions \eqref{eq:cnd_L0}-\eqref{eq:cnd_L3} from Example~\ref{example:1}, and
$$\widehat{\alpha}=\int_{\Theta} \frac{d\theta}{|L(\theta)|}=\int_0^1 \frac{d\theta}{1-s\theta}=-\frac{\ln(1-s)}{s}.$$

Consequently, in view of the analysis in Example~\ref{example:1}, the hypotheses of Proposition~\ref{prop:affine_general} (or rather Corollary~\ref{prop:affine}) are satisfied in this setting whenever $\widehat{\alpha}|r|<1$, which is equivalent~to
\begin{equation}\label{e:sr}
|r|<-\frac{s}{\ln(1-s)}.
\end{equation}
More precisely, \eqref{e:sr} implies \eqref{ex:sum3}, which in turn yields \eqref{ex:sum}. In addition, we see that
$$\widehat{R}=\int_{\Theta} L(\theta)\,d\theta=\int_0^1 (1-s\theta)\,d\theta=1-\frac{s}{2}<1.$$

We have therefore shown that the Markov chain $\Phi^{\bullet}$ defined via $S_{\theta}^{(n)}$ in \eqref{ex1:S} admits a~geometrically attracting measure with a finite first moment whenever \eqref{e:sr} holds, which is the case, in particular, when $|r|\leq 1-s$ (since $\widehat{\alpha}< (1-s)^{-1}$). 
\end{example}

\begin{example}
Define $(\Theta,\mathcal{A},\vartheta)$ as in the previous example, and let \hbox{$T_n:\Theta\to \Theta$}, $n\in\n$, be any sequence of measurable transformations that preserve the Lebesgue measure on $(0,1)$; e.g., one can take $T_n:=T^n$, $n\in\n$, where $T$ is the (doubling) baker's map, given by
$$T(\theta)=
\begin{cases}
2\theta, &\text{for}\quad \theta\in (0,1/2),\\
2\theta-1 &\text{for} \quad \theta\in [1/2,1).
\end{cases}
$$
Moreover, let $Q(\theta)$, $\theta\in\Theta$, be arbitrary real orthogonal $N\times N$ matrices such that $\theta\mapsto Q(\theta)$ is $\mathcal{A}/\mathcal{B}(M_N)$-measurable. Given this, on $X=\mathbb{R}^N$, consider the transformations
$$S_{\theta}^{(n)}(x):=a\left(1-\frac{1}{n^s\,\mathcal{C}(\theta)}\right)(1-T_n(\theta))^{1/2}Q(T_n(\theta))x+r^n\mathcal{D}(\theta)f(T_n(\theta)),\;\; x\in X,\; \theta\in\Theta,\; n\in\n,$$
where $0<|a|<\sqrt{e}$, $s\geq 0$, $r\in \mathbb{R}$, and $f:\Theta\to X$, $\mathcal{C},\mathcal{D}:\Theta\to\mathbb{R}$ are any measurable functions such that~$f,\mathcal{D}$ are bounded, and $C:=\inf_{\theta\in\Theta}\mathcal{C}(\theta)>1$.

Observe that each $S_{\theta}^{(n)}$ is of the form \eqref{e:af_Tn} with $\wA$, $\widehat{b}$, $\{\mathcal{C}_n\}_{n\in\n}$, and $\{\mathcal{D}_n\}_{n\in\n}$ defined by
$$\wA(\theta)(x)=a(1-\theta)^{1/2}Q(\theta)x,\quad \widehat{b}(\theta)=f(\theta)\quad\text{for}\quad\theta\in\Theta,$$
$$\mathcal{C}_n(\theta):=1-\frac{1}{n^s\mathcal{C}(\theta)},\quad 
\mathcal{D}_n(\theta)=r^n\mathcal{D}(\theta)\quad\text{for}\quad \theta\in\Theta,\; n\in\n,$$
where $\mathcal{C}_n$ and $\mathcal{D}_n$ satisfy \eqref{eq:cnd_L4} with
$$c_n:=1-\frac{1}{n^sC},\quad d_n=D|r|^n \quad\text{for}\quad n\in\n,\quad\text{where}\quad D:=\sup_{\theta\in \Theta} |\mathcal{D}(\theta)|.$$
In addition, the mappings $\wA,\widehat{b}$, and $L:=|a|(1-\operatorname{id}_{\Theta})^{1/2}$ belong to the class introduced in \hbox{Example~\ref{example:2}} (with $p=1/2$, $q=\bar{p}=\bar{q}=0$, $F\equiv a$, $\mathbb{B}(\theta)(x)=Q(\theta)x$, and the given $f$), and thereby satisfy \eqref{eq:cnd_L0}-\eqref{eq:cnd_L3}. Furthermore,
$$\widehat{\alpha}=\int_{\Theta}\frac{d\theta}{|L(\theta)|}=\frac{1}{|a|}\int_0^1 \frac{d\theta}{\sqrt{1-\theta}}=\frac{2}{|a|}\quad\text{and}\quad \underline{c}:=\inf_{n\in\n} c_n=1-\frac{1}{C}>0.$$

Relying on condition \eqref{ex:sum}, we now verify when Proposition~\ref{prop:affine_general} (or more precisely, Corollary \ref{prop:affine}) can be applied in this setting, depending on $s$ and $r$. First of all, wee see that
$$\sum_{n=1}^{\infty} \widehat{\alpha}^n d_n =D\sum_{n=1}^{\infty} \left(\frac{2|r|}{|a|}\right)^n<\infty\;\;\text{iff}\;\; |r|<\frac{|a|}{2},$$
which means that $|r|<|a|/2$ is necessary for \eqref{ex:sum} to hold. Next, note that
$$\sum_{n=1}^{\infty}\left(1-c_n\right)=\frac{1}{C}\sum_{n=1}^{\infty}\frac{1}{n^s}<\infty\;\;\;\text{iff}\;\;\;s>1,$$
$$\sum_{n=1}^{\infty} \left(\frac{\widehat{\alpha}}{\underline{c}}\right)^n d_n=D\sum_{n=1}^{\infty}\left(\frac{2|r|}{|a|(1-1/C)}\right)^n<\infty\;\;\;\text{iff}\;\;\; 
|r|<\frac{|a|}{2}\left(1-\frac{1}{C}\right).$$
Hence, in the case where $s>1$, condition \eqref{ex:sum} is ensured by \eqref{ex:sum3} whenever $|r|<|a|/2$. If~$s\in [0,1]$, then \eqref{ex:sum3} fails, but one can use instead  \eqref{ex:sum1} (which also yields \eqref{ex:sum}), provided that \hbox{$|r|<|a|(1-1/C)/2$}. Consequently, by the discussion in Example~\ref{example:1}, we infer that the hypotheses of Proposition~\ref{prop:affine_general} are fulfilled whenever either $s>1$ and $|r|<|a|/2$, or $0\leq s\leq 1$ and $|r|<|a|(1-1/C)/2$. What is more, 
$$\widehat{R}=\int_{\Theta} L(\theta)\,d\theta=|a|\int_0^1 \sqrt{1-\theta}\,d\theta=\frac{2|a|}{3},$$
which shows that the geometrically attracting measure $\pi$ of the chain $\Phi^{\bullet}$, existing in these cases, belongs to $\mathcal{M}_{1,1}(X)$ whenever $|a|<3/2$.
\end{example}

Regarding backward iterates, an interpretation of Theorem \ref{thm:main2} (supplemented with Proposition \ref{cor:finite_moment}) in the affine setting is immediate and can be stated as follows.

\begin{proposition}\label{prop:affine_general2}
Let \hbox{$\mathbb{A}_n:\Theta_n\to X^X$}, \hbox{$b_n:\Theta_n\to X$}, and $L_n:\Theta_n\to (0,\infty)$, $n\in\n$, be such that, for every $n\in\n$, the functions $\A{n}{\theta}:X\to X$, $\theta\in\Theta_n$, $n\in\n$, are continuous linear operators, and the mappings $b_n$, $L_n$, and $X\times \Theta_n \ni (x,\theta) \mapsto \A{n}{\theta}(x)\in X$ are measurable. Furthermore, assume that conditions \ref{cnd:E1} and \ref{cnd:E2} of Proposition \ref{prop:affine_general} hold, and that, for some $n_0\in\n$,
$$\sup_{n\geq n_0} \int_{\Theta_n} \norma{b_n(\theta)}\,\vartheta_n(d\theta)<\infty.$$
Then the conclusion of Theorem \ref{thm:main2} is valid for the process $\Psi^{\bullet}$ defined in \eqref{e:model_def_rev} via the affine transformations of the form \eqref{def:iso}. Moreover, the attracting measure $\bar{\pi}$ satisfying \eqref{e:main2} belongs to $\mathcal{M}_{1,1}(X)$ whenever $n_0=1$ and $R<1$ (where $R$ is defined in \emph{\ref{cnd:E2}}).
\end{proposition}

As can be seen, the assumptions in this case are substantially weaker than those in Proposition~\ref{prop:affine_general}, which allows one to identify a~much broader class of admissible families of transformations. For instance, building on Example \ref{example:1}, we may characterize the following class:

\begin{example}
Let $(\Theta,\mathcal{A},\vartheta)$ and $T_n:\Theta_n\to \Theta$, $n\in\n$, be such as in Example \ref{example:1}. Further, suppose that we are given measurable functions $\widehat{b}:\Theta\to X$, $L:\Theta\to (0,\infty)$, and continuous linear operators $\mathbb{A}_n(\theta):X \to X$, $\theta\in\Theta$, such the mapping $X\times \Theta \ni (x,\theta) \mapsto \wA(\theta)(x)\in X$ is measurable, conditions \eqref{eq:cnd_L1}, \eqref{eq:cnd_L2} hold, and $$\int_{\Theta}\|\widehat{b}(\theta)\|\,\vartheta(d\theta)<\infty.$$ Additionally, let $\mathcal{C}_n, \mathcal{D}_n:\Theta_n\to \mathbb{K}$, $n\in\n$, be any measurable functions such that
$$|\mathcal{C}_n(\theta)|\leq 1\quad\text{and}\quad |\mathcal{D}_n(\theta)|\leq D \quad\text{for all}\quad \theta\in\Theta_n,\;n\in\n,$$
with some $D\geq 0$. Then the conclusion of Proposition \ref{prop:affine_general2} is valid for the process $\Psi^{\bullet}$ defined via the transformation in \eqref{e:af_Tn} with $R=\widehat{R}$, where $\widehat{R}$ is defined in \eqref{eq:cnd_L2}.
\end{example}

\section{Proofs of the main results}\label{sec:5}
In this part of the paper, we will prove Theorems \ref{thm:main} and \ref{thm:main2}. To do so, we shall first present some auxiliary results.

\begin{lemma}\label{lem:composition}
Let $T_1,T_2,\ldots:X\to X$ be Lipschitz continuous transformations with constants $L_1,L_2,\ldots$, respectively, and define
$$\bmathcal{T}_0=\mathcal{T}_0:=\operatorname{id}_X,\quad\mathcal{\bmathcal{T}}_n:=T_1\circ\ldots\circ T_n,\quad\text{and}\quad \mathcal{\mathcal{T}}_n:=T_n\circ\ldots\circ T_1\quad\text{for}\quad n\in\n.$$
 Then, for any $n\geq 2$ and $x\in X$,
\begin{equation}\label{e:comp1}
\rho(\bmathcal{T}_{n-1}(x),\, \bmathcal{T}_n(x))\leq \rho\left(x,T_n(x)\right)\prod_{j=1}^{n-1} L_j.
\end{equation}
Moreover, if $T_1,T_2,\ldots$ are also invertible, then for every $n\geq 2$,
\begin{equation}\label{e:comp2}
\rho(\mathcal{T}_{n-1}(x),\, \mathcal{T}_n(x))\leq 
\rho\left(x, (\mathcal{T}_{n-1}^{-1}\circ T_n\circ \mathcal{T}_{n-1})(x)\right) \prod_{j=1}^{n-1} L_j.
\end{equation}
\end{lemma}
\begin{proof}
Let $n\geq 2 $ and $x\in X$. 

It easily seen that $\bmathcal{T}_{n-1}$ is Lipschitz continuous with constant~$\prod_{j=1}^{n-1} L_j$. Hence,
$$
\rho(\bmathcal{T}_{n-1}(x),\, \bmathcal{T}_n(x))=\rho(\bmathcal{T}_{n-1}(x),\, \bmathcal{T}_{n-1}(T_n(x)))\leq \rho(x,T_n(x))\prod_{j=1}^{n-1} L_j,
$$
which yields \eqref{e:comp1}. 

Now assume that $T_1, T_2,\ldots$ are invertible. To establish \eqref{e:comp2}, we first note that for any $F:X\to X$ and $m\in\n$,
\begin{align}
\begin{split}\label{e:indu0}
\rho(\mathcal{T}_m(x),\, F(x))&=\rho\left(T_m\left(\mathcal{T}_{m-1}(x)\right),\, T_m\left((T_m^{-1}\circ F)(x)\right)\right)\\
&\leq \rho\left(\mathcal{T}_{m-1}(x),\, \left(T_m^{-1}\circ F\right)(x)\right)L_m.
\end{split}
\end{align}
By induction, this implies that, for every $k\in\n$,
\begin{equation}\label{e:indu}
\rho(\mathcal{T}_k(x),\, F(x))\leq \rho\left(x,\, (\mathcal{T}_k^{-1}\circ F )(x)\right) \prod_{j=1}^k L_j \quad \text{with any}\quad F:X\to X.
\end{equation}
Indeed, for $k=1$, the inequality follows immediately from \eqref{e:indu0} (applied with $m=1$). Suppose that \eqref{e:indu} is satisfied for some arbitrarily fixed $k\in\n$. Then, applying \eqref{e:indu0} with $m=k+1$ and the inductive hypothesis with $T_{k+1}^{-1}\circ F$ in place of $F$, we obtain
\begin{align*}
\rho\left(\mathcal{T}_{k+1}(x),\, F(x)\right)&\leq \rho\left(\mathcal{T}_k(x),\,(T_{k+1}^{-1}\circ F)(x) \right)L_{k+1}
\\&\leq \rho\left(x,\,(\mathcal{T}_k^{-1}\circ T_{k+1}^{-1}\circ F)(x) \right)\left(\prod_{j=1}^k L_j\right)L_{k+1}= \rho\left(x,\,(\mathcal{T}_{k+1}^{-1}\circ F)(x) \right)\prod_{j=1}^{k+1} L_j.
\end{align*}
Finally, taking $k=n-1$ and $F=\mathcal{T}_n\,(=T_n\circ\mathcal{T}_{n-1})$ in \eqref{e:indu} completes the proof.
\end{proof}

The next result offers a~simple observation concerning an estimation of the bounded Lipschitz distance.
\begin{lemma}\label{lem:rand_var_dist}
Let $\xi$ and $\zeta$ be any two $X$-valued random variables with distributions $\mu_{\xi}$ and~$\mu_{\zeta}$, respectively, such that
$$\pr(\rho(\xi,\, \zeta)>\alpha)\leq \beta$$
for certain $\alpha, \beta\in\mathbb{R}_+$. Then
$$d_{BL}(\mu_{\xi},\, \mu_{\eta})\leq \alpha+2\beta.$$
\end{lemma}
\begin{proof}
Let $H:=\{\rho(\xi,\zeta)>\alpha\}$. Then, for every $f\in L_{b,1}(X)$,
\begin{align*}
|\<f,\mu_{\xi}\> - \<f,\mu_{\zeta}\>|&\leq \int_{\Omega} |f(\xi)-f(\zeta)|\,d\pr
=\int_{\Omega\backslash H} |f(\xi)-f(\zeta)|\,d\pr+\int_H |f(\xi)-f(\zeta)|\,d\pr\\
&\leq \int_{\Omega\backslash H} \rho(\xi,\zeta)\,d\pr+2\pr(H)\leq \alpha+2\beta,
\end{align*}
which implies the claim.
\end{proof}
The core argument employed in the proofs of our main results relies on \hbox{\cite[Lemma 5.2]{b:DiaconisFreedman1999}}, which we quote below. For the sake of completeness, in the \hyperref[sec:appendix]{Appendix}, we also include a~detailed proof of this lemma.

\begin{lemma}\label{lem:diaconis}
Let $\{\xi_n\}_{n\in\n}$ be a~sequence of mutually independent and identically distributed random variables with values in $[-\infty, \infty)$. Further, suppose that there exist constants $\alpha\geq 0$ and $\beta>0$ such that
\begin{equation}
\label{e:czeb_eq}
\pr(\xi_1>t)\leq \alpha e^{-\beta t} \quad \text{for all}\quad t>0.
\end{equation}
Then $\ew\,\xi_1$ exists, $-\infty\leq \ew\,\xi_1<\infty$, and for every constant $c>\ew\xi_1$ there is $r\in (0,1)$ such that
\begin{equation}
\label{e:diac_aasert}
\pr\left(\sum_{j=1}^n \xi_j >nc \right)\leq r^n\quad\text{for all}\quad n\in\n.
\end{equation}
\end{lemma}

\begin{remark}\label{rem:czeb}
It follows directly form the Chebyshev inequality that, if $\bar{\xi}$ is a~random variable with positive values and $\ew\bar{\xi}<\infty$, then \eqref{e:czeb_eq} holds for $\xi_1=\ln\,\bar{\xi}$ with $\alpha=\ew\,\bar{\xi}$ and $\beta=1$.
\end{remark}

We are now ready to prove the main results of Sections \ref{sec:2} and \ref{sec:3}.

\begin{proof}[Proof of Theorem \ref{thm:main}]
For convenience of notation, if $n_0>1$, we additionally set $$M_n(\theta_1,\ldots,\theta_n):=\delta\quad\text{for}\quad  \theta_1\in\Theta_1,\ldots,\theta_n\in\Theta_n,\;\, n\in\{1,\ldots,n_0-1\}.$$

Consider the random variables
\begin{gather*}
\bar{\xi}_n:=L_n(\eta_n),\quad \xi_n:=\ln \bar{\xi_n},\quad \text{and}\quad \kappa_n:=M_n(\eta_1,\ldots,\eta_n)\quad\text{for}\quad n\in\n,
\end{gather*}
with $\{\eta_n\}_{n\in\n}$ defined as in \eqref{df:eta}. In view of \ref{cnd:A2}, it is clear that  $\{\bar{\xi}_n\}_{n\in\n}$ and $\{\xi_n\}_{n\in\n}$ are sequences of mutually independent and identically distributed random variables, and that $\ew\bar{\xi}_1$ is finite. According to Remark~\ref{rem:czeb}, the latter implies \eqref{e:czeb_eq}, which, in turn, guarantees that  $\{\xi_n\}_{n\in\n}$ fulfills the assumptions of Lemma~\ref{lem:diaconis}. This, in particular, ensures that $\ew \xi_1$ is well-defined (though not necessarily finite), and, due to \ref{cnd:A2}, we have $\ew \xi_1<0$. Moreover, from \ref{cnd:A4} it follows that $0\leq \sup_{n\in\n} \ew\kappa_n= \delta<\infty$.

Let us now fix $\gamma\in (0,-\ew\,\xi_1)$ and $s\in (1,e^{\gamma})$. Then, by virtue of Lemma~\ref{lem:diaconis} (applied to the sequence $\{\xi_n\}_{n\in\n}$ with~$c=-\gamma$), there exists $r\in (0,1)$ such that
\begin{equation}\label{e:proof1}
\pr\left(\sum_{j=1}^k \xi_j >-k\gamma\right)\leq r^k\quad \text{for all}\quad k\in\n.
\end{equation}

Further, for every $k\in\n$, define the events
$$
A_k:=\left\{\prod_{j=1}^k \bar{\xi}_j>e^{-k\gamma} \right\}=\left\{ \omega=(\theta_1,\theta_2,\ldots)\in\Omega:\; \prod_{j=1}^k L_j(\theta_j)>e^{-k\gamma}\right\},\vspace{-0.2cm}
$$

$$B_k:=\left\{\kappa_k>s^k\right\}=\left\{ \omega=(\theta_1,\theta_2,\ldots)\in\Omega:\; M_k(\theta_1,\ldots,\theta_k)>s^k\right\},
$$
and note that 
\begin{equation}\label{e:proof2}
\pr(A_k)\leq r^k,\quad \quad \pr(B_k)\leq \delta s^{-k}\quad\text{for all}\quad k\in\n.
\end{equation}
Indeed, the first inequality follows from \eqref{e:proof1}, since
$$
\pr(A_k)=\pr\left(\ln\left( \prod_{j=1}^k \bar{\xi}_j\right)>-k\gamma\right)=\pr\left(\sum_{j=1}^k \xi_j>-k\gamma \right)\leq r^k\quad\text{for}\quad k\in\n,$$
while the second one is an immediate consequence of the Chebyshev inequality. 

Referring to condition \ref{cnd:A1}, for every $k\in\n$, one may find~$\bar{N}_k\in\mathcal{A}_k$ with $\vartheta_k(\bar{N}_k)=0$ such that, for any $\theta\in \Theta_k\backslash \bar{N}_k$, the map $S_{\theta}^{(k)}$ is Lipschitz continuous with constant $L_k(\theta)$. Moreover, assumption  \ref{cnd:A3} implies that, for every $k\geq n_0$, there exists $\tilde{N}_k\in\bigotimes_{j=1}^k\mathcal{A}_j$ with $\bigotimes_{j=1}^k \vartheta_j(\tilde{N}_k)=0$ such that, for any $(\theta_1,\ldots,\theta_k)\in (\prod_{j=1}^k\Theta_j) \backslash \tilde{N}_k$, the transformations $S_{\theta_1}^{(1)},\ldots,S_{\theta_k}^{(k)}$ are invertible, and
\begin{equation}\label{e:A3}
\rho\left(x_0,\, \left(\mathcal{S}_{\theta_1,\dots,\theta_{k-1}}^{-1}\circ S_{\theta_k}^{(k)} \circ \mathcal{S}_{\theta_1,\dots,\theta_{k-1}}\right) (x_0)\right)\leq M_k(\theta_1,\ldots,\theta_k).
\end{equation}
Given this, for every $k\in\n$, define
$$\bar{N}_k^{\infty}:=\prod_{j=1}^{\infty} E_j \;\;\text{with} \;\; E_k=\bar{N}_k\;\;\text{and}\;\; E_j=\Theta_j \;\;\text{for}\;\;j\neq k,\quad \tilde{N}_k^{\infty}:=\tilde{N}_{k+n_0-1} \times \prod_{j=k+n_0}^{\infty}\Theta_j,$$ and set $N:=\bigcup_{k=1}^{\infty} \bar{N}_k^{\infty}\cup \tilde{N}_k^{\infty}$. Then $N\in\mathcal{F}$, $\pr(N)=0$, and it clear that, if $(\theta_1,\theta_2,\ldots)\in\Omega\backslash N$, then\vspace{-0.2cm}
\begin{equation}\label{e:lip_inv}
S_{\theta_k}^{(k)}\; \text{is both}\; L_k(\theta_k)\text{-Lipschitz continuous and invertible for every}\;k\in\n,
\end{equation}
and \eqref{e:A3} holds for all $k\geq n_0$.

Let us now introduce
$$C_n:=N\cup \bigcup_{k=n}^{\infty}\left (A_k\cup B_k\right) \quad \text{for}\quad n\in\n,$$
which are obviously $\mathcal{F}$-measurable. Appealing to \eqref{e:proof2}, we see that
$$
\pr(C_n)\leq  \sum_{k=n}^{\infty}\pr(A_k) + \sum_{k=n}^{\infty}\pr(B_k)
\leq \sum_{k=n}^{\infty}r^k+\delta\sum_{k=n}^{\infty}\left(\frac{1}{s}\right)^k=\frac{r^n}{1-r}+\frac{\delta s}{s-1}\left(\frac{1}{s}\right)^n.
$$
Thus, there exist $r_1\in(0,1)$ and $c_1>0$ such that
\begin{equation}\label{e:proof3}
\pr(C_n)\leq c_1r_1^n\quad\text{for all}\quad n\in\n.
\end{equation}

We will show that there exist $r_2\in (0,1)$ and $c_2>0$ such that
\begin{align}
\begin{split}\label{e:proof4}
&\rho\big(\Phi_n^{x}(\omega), \Phi_n^{x_0}(\omega)\big)\leq r_2^n V(x)\quad \text{and}\quad   \rho\left(\Phi_n^{x_0}(\omega), \Phi_{n+m}^{x_0}(\omega)\right)\leq c_2r_2^n    \\[0.2cm]
&\text{for any}\;\;\,  x\in X,\; n\geq n_0,\; m\in\n,\; \omega\in \Omega\backslash C_n.
\end{split}
\end{align}
To this end, fix arbitrary $x\in X$, $n\geq n_0$, $m\in\n$, and (assuming that $\Omega\backslash C_n\neq \emptyset$) let \hbox{$\omega=(\theta_1,\theta_2,\ldots)\in \Omega\backslash C_n$}. From the definition of~$C_n$, it follows that \eqref{e:lip_inv} holds with the coordinates of $\omega$, and that
\begin{equation}\label{e:proof5}
\prod_{j=1}^k L_j(\theta_j)\leq e^{-k\gamma}\quad \text{and}\quad M_k(\theta_1,\ldots,\theta_k)\leq s^k\quad\text{for any}\quad k\geq  n.
\end{equation}
Taking into account \eqref{e:model_def_exp}, we can apply this observation as follows. The Lipschitz part of \eqref{e:lip_inv}, together with the first inequality in~\eqref{e:proof5}, gives
\begin{align*}
\rho\left(\Phi_n^{x}(\omega), \Phi_{n}^{x_0}(\omega)\right)&=\rho\left(\mathcal{S}_{\theta_1,\ldots,\theta_n}(x),\, \mathcal{S}_{\theta_1,\ldots,\theta_n}(x_0)\right)
\leq \left(\prod_{j=1}^n L_j(\theta_j)\right) \rho(x,x_0)
\\
& \leq e^{-n\gamma}V(x)\leq \left(se^{-\gamma}\right)^n V(x).
\end{align*}
Meanwhile, since the second assertion of Lemma \ref{lem:composition} holds via \eqref{e:lip_inv}, combining it with the validity of \eqref{e:A3} for $k\geq n_0$ and both inequalities in \eqref{e:proof5}, yields
\begin{align*}
\rho\left(\Phi_n^{x_0}(\omega), \Phi_{n+m}^{x_0}(\omega)\right)&=
\rho\left(\mathcal{S}_{\theta_1,\ldots,\theta_n}(x_0),\, \mathcal{S}_{\theta_1,\ldots,\theta_{n+m}}(x_0)\right)\\
&\leq \sum_{k=n+1}^{n+m} \rho\left(\mathcal{S}_{\theta_1,\ldots,\theta_{k-1}}(x_0),\, \mathcal{S}_{\theta_1,\ldots,\theta_k}(x_0)\right)\\
&\leq \sum_{k=n+1}^{n+m}\rho\left(x_0,\; \left(\mathcal{S}_{\theta_1,\ldots,\theta_{k-1}}^{-1}\circ S_{\theta_k}^{(k)}\circ \mathcal{S}_{\theta_1,\ldots,\theta_{k-1}}\right)(x_0) \right) \prod_{j=1}^{k-1}L_j(\theta_j) \\
&\leq \sum_{k=n+1}^{n+m} M_k(\theta_1,\ldots,\theta_k)\prod_{j=1}^{k-1}L_j(\theta_j)\leq \sum_{k=n+1}^{\infty} s^k e^{-(k-1)\gamma}\\
&=e^{\gamma}\sum_{k=n+1}^{\infty}(se^{-\gamma})^k=e^{\gamma}\frac{(se^{-\gamma})^{n+1} }{1-se^{-\gamma}}=\frac{s}{1-se^{-\gamma}}(se^{-\gamma})^n.
\end{align*}
We have therefore shown that \eqref{e:proof4} is satisfied with 
\begin{equation}\label{e:rc}
r_2:=se^{-\gamma}\in(0,1)\quad\text{and}\quad c_2:=s(1-se^{-\gamma})^{-1}>0.
\end{equation}

Observations \eqref{e:proof3} and \eqref{e:proof4} imply that
$$
\max\left\{ \pr\big(\rho\left(\Phi_n^{x}, \Phi_n^{x_0}\big)> r_2^n V(x) \right),\; \pr\left(\rho\left(\Phi_n^{x_0}, \Phi_{n+m}^{x_0}\right)> c_2r_2^n \right)  \right\}\leq \pr(C
_n)\leq c_1 r_1^n
$$
for all $x\in X$, $n\geq n_0$, and $m\in\n$. This, in turn, according to Lemma \ref{lem:rand_var_dist}, shows that
\begin{equation} \label{e:proof6}
d_{BL}\left(\delta_{x} P^{(n)},\; \delta_{x_0} P^{(n)} \right)\leq r_2^n V(x)+2c_1r_1^n\quad\text{for all} \quad x\in X,\;  n\geq n_0,
\end{equation}
and
\begin{equation} \label{e:proof7}
d_{BL}\left(\delta_{x_0} P^{(n)},\; \delta_{x_0} P^{(n+m)} \right)\leq c_2r_2^n +2c_1r_1^n\quad\text{for all} \quad n\geq n_0,\;m\in\n.
\end{equation}

Since $r_1,r_2\in (0,1)$, it follows from \eqref{e:proof7}  that $\{\delta_{x_0} P^{(n)} \}_{n\in\n}$ is a~Cauchy sequence in $(\mathcal{M}_1(X), d_{BL})$. Consequently, there exists a~measure $\pi\in\mathcal{M}_1(X)$ such that
$$\lim_{n\to\infty} d_{BL}\left(\delta_{x_0} P^{(n)}, \pi\right)=0.$$
Moreover, after passing to the limit as $m\to \infty$ in \eqref{e:proof7}, we see that
\begin{equation} \label{e:proof8}
d_{BL}\left(\delta_{x_0} P^{(n)},\pi\right)\leq c_2 r_2^n +2c_1 r_1^n\quad\text{for every}\quad n\geq n_0.
\end{equation}
Thus, by combining \eqref{e:proof6} and \eqref{e:proof8}, we can conclude that, for any $x\in X$ and $n\geq n_0$,
\begin{align*}
d_{BL}\left(\delta_x P^{(n)},\pi\right)&\leq  d_{BL}\left(\delta_x P^{(n)},\delta_{x_0} P^{(n)}\right)+d_{BL}\left(\delta_{x_0} P^{(n)},\pi\right)\\
&\leq r_2^n (V(x)+c_2) +4c_1r_1^n\leq \max\{r_1,r_2\}^n(V(x)+c_2+4c_1).
\end{align*}
Taking $q:=\max\{r_1,r_2\}$ and $C_0:=1+c_2+4c_1$ therefore gives
$$
d_{BL}\left(\delta_x P^{(n)},\pi\right)\leq C_0 q^n(V(x)+1)\quad\text{for all}\quad x\in X,\;n\geq n_0.
$$
However, since $d_{BL}(\mu_1,\mu_2)\leq 2$ for any $\mu_1,\mu_2\in\mathcal{M}_1(X)$, we can replace $C_0$ with the constant \hbox{$C:=\max\{2,C_0\}/q^{n_0}$} to obtain
$$
d_{BL}\left(\delta_x P^{(n)},\pi\right)\leq C q^n(V(x)+1)\quad\text{for all}\quad x\in X,\;n\in\n.
$$

Finally, in order to deduce \eqref{e:main}, it suffices to observe that
$$d_{BL}\left(\mu P^{(n)},\, \pi\right)\leq \int_X d_{BL}\left(\delta_x P^{(n)},\, \pi\right)\mu(dx)\quad\text{for all}\quad \mu\in\mathcal{M}_1(X),\; n\in\n,$$
which follows directly from \eqref{e:attracting_dirac}. Clearly, due to Remark \ref{rem:3}, $\pi$ is then also attracting in the sense of \eqref{def:attracting}. The proof is therefore complete.
\end{proof}

\begin{proof}[Proof of Theorem \ref{thm:main2}]
Define 
$$K_n(\theta):=
\begin{cases}
\Delta &\text{for}\quad \theta\in\Theta_n,\; n\in\{1,\ldots,n_0-1\},\\
\rho\left(x_0,\,S_{\theta}^{(n)}(x_0)\right) &\text{for}\quad \theta\in\Theta_n,\;n\geq n_0.
\end{cases}
$$
Proceeding analogously to the proof of Theorem \ref{thm:main}, under hypotheses \ref{cnd:A1}, \ref{cnd:A2}, and~\ref{cnd:B3}, one can construct a~sequence of $\mathcal{F}$-measurable sets $\{C_n\}_{n\in \n}$ so that \eqref{e:proof3} is satisfied for some $c_1>0$ and $r_1\in(0,1)$, and moreover, for any $n\geq n_0$ and $(\theta_1,\theta_2,\ldots)\in\Omega\backslash C_n$, the following properties hold with certain $\gamma>0$ and $s\in(1,e^\gamma)$:
\begin{equation}
\begin{gathered}\label{e:proof9}
S_{\theta_k}^{(k)}\;\,\text{is}\;\,L_k(\theta_k)\text{-Lipschitz continuous}\;\;\text{for every}\;\; k\in\n,\\
\quad\prod_{j=1}^k L_j(\theta_j)\leq e^{-k\gamma},\quad\text{and}\quad K_k(\theta_k)\leq s^k\quad\text{for all}\quad k\geq  n.
\end{gathered}
\end{equation}
The procedure differs only in the construction of the \hbox{$\mathbb{P}$-null} set~$N$ and in the definition of the random variables~$\kappa_k$ used to define $B_k$, $k\in\n$. Specifically, $N\in\mathcal{F}$ needs to be constructed so that $S_{\theta_k}^{(k)}$ is $L_k(\theta_k)$-Lipschitz continuous for every $(\theta_1,\theta_2,\ldots)\in \Omega\backslash N$, while $\kappa_k$ must be given by $\kappa_k:=K_k(\eta_k)$ for $k\in\n$. Then $\pr(B_k)\leq \Delta s^{-k}$ for all $k\in \n$, since $\sup_{k\in\n} \ew \kappa_k=\Delta$ due to condition \ref{cnd:B3}.

Now fix arbitrary $x\in X$, $n\geq n_0$, $m\in\n$, and $\omega=(\theta_1,\theta_2,\ldots)\in \Omega\backslash C_n$. Referring to~\eqref{e:proof9} and the first assertion of Lemma \ref{lem:composition}, we conclude, similarly as in the proof of the previous theorem, that
$$\rho\left(\Psi_n^{x}(\omega), \Psi_{n}^{x_0}(\omega)\right)=\rho\left(\bmathcal{S}_{\theta_1,\ldots,\theta_n}(x),\, \bmathcal{S}_{\theta_1,\ldots,\theta_n}(x_0)\right)
\leq \left(\prod_{j=1}^n L_j(\theta_j)\right) \rho(x,x_0)\leq r_2^n V(x),
$$
and that
\begin{align*}
\rho\left(\Psi_n^{x_0}(\omega), \Psi_{n+m}^{x_0}(\omega)\right)&=
\rho\left(\bmathcal{S}_{\theta_1,\ldots,\theta_n}(x_0),\, \bmathcal{S}_{\theta_1,\ldots,\theta_{n+m}}(x_0)\right)\leq \sum_{k=n+1}^{n+m} \rho\left(\bmathcal{S}_{\theta_1,\ldots,\theta_{k-1}}(x_0),\, \bmathcal{S}_{\theta_1,\ldots,\theta_k}(x_0)\right)\\
&\leq \sum_{k=n+1}^{n+m}\rho\left(x_0,\,S_{\theta_k}^{(k)}(x_0)\right) \prod_{j=1}^{k-1}L_j(\theta_j)  \overset{k\geqslant n_0}{=}\sum_{k=n+1}^{n+m} K_k(\theta_k)\prod_{j=1}^{k-1}L_j(\theta_j)\\
&\leq \sum_{k=n+1}^{\infty} s^k e^{-(k-1)\gamma}=c_2r_2^n,
\end{align*}
where $c_2>0$ and $r_2\in (0,1)$ are defined as in \eqref{e:rc}. This yields
$$
\max\left\{ \pr\big(\rho\left(\Psi_n^{x}, \Psi_n^{x_0}\big)> r_2^n V(x) \right),\; \pr\left(\rho\left(\Psi_n^{x_0}, \Psi_{n+m}^{x_0}\right)> c_2r_2^n \right)  \right\}\leq \pr(C_n)\leq c_1 r_1^n.
$$

Consequently, an application of Lemma \ref{lem:rand_var_dist}, yields
$$
d_{BL}\left(\delta_{x} \bar{P}^{(n)},\; \delta_{x_0} \bar{P}^{(n)} \right)\leq r_2^n V(x)+2c_1r_1^n\quad\text{for all} \quad x\in X,\;  n\geq n_0
$$
and
$$
d_{BL}\left(\delta_{x_0} \bar{P}^{(n)},\; \delta_{x_0} \bar{P}^{(n+m)} \right)\leq c_2r_2^n +2c_1r_1^n\quad\text{for all} \quad n\geq n_0,\;m\in\n.
$$
The remainder of the argument is identical to the proof of Theorem \ref{thm:main}, with $P^{(n)}$ replaced by $\bar{P}^{(n)}$.

\end{proof}

\appendix
\setcounter{secnumdepth}{1}
\renewcommand{\theequation}{A.\arabic{equation}}
\renewcommand{\thesection}{A}

\section*{Appendix}\label{sec:appendix}
For the self-containedness of the paper and reader's convenience, we provide here a~detailed proof of Lemma \ref{lem:diaconis}, closely following the original argument from \cite{b:DiaconisFreedman1999}. Before proceeding, we first establish two technical facts that are essential in this proof.

In what follows, when considering $e^{\lambda\xi}$ with $\lambda > 0$ and a~random variable $\xi$ taking values in $[-\infty,\infty)$, we adopt the convention $e^{-\infty}:=0$.

\begin{lemma}\label{lem:diac1}
Let $\xi$ be a~random variable with values in $[-\infty,\infty)$ that satisfies \eqref{e:czeb_eq} for certain $\alpha\geq 0$ and $\beta>0$. Then, for every $\lambda\in (0,\beta)$,
\begin{enumerate}[label=\textnormal{(\roman*)}, leftmargin=*]
\item\label{cnd:d1} $\ew\, e^{\lambda \xi}<\infty$ and, in particular, $\ew\xi$ exists with $-\infty\leq \ew \xi<\infty$;
\item\label{cnd:d2} if $\xi$ is a.s. bounded from below, then $\ew e^{\lambda |\xi|}<\infty$, and, in particular, $\ew|\xi|<\infty$.
\end{enumerate}
\end{lemma}
\begin{proof}
Let $\lambda\in (0,\beta)$. It is clear that $\ew\left[e^{\lambda\xi}\mathbbm{1}_{\{\xi\leq 0\}} \right]\leq \pr(\xi\leq 0)\leq 1$.
Using condition \eqref{e:czeb_eq} we can, in turn, conclude that
\begin{align}
\label{e:lem_d1}
\begin{split}
\ew\left[e^{\lambda\xi}\mathbbm{1}_{\{\xi>0\}} \right]
&=\sum_{n=0}^{\infty} \left[e^{\lambda\xi}\mathbbm{1}_{\{n<\xi\leq n+1\}} \right]\leq \sum_{n=0}^{\infty} e^{\lambda(n+1)}\pr(\xi>n)\\
&=e^{\lambda}\pr(\xi>0)+\sum_{n=1}^{\infty} e^{\lambda (n+1)}\pr(\xi>n)\leq e^{\lambda}\left(1+\alpha \sum_{n=1}^{\infty} e^{(\lambda-\beta)n}\right)<\infty.
\end{split}
\end{align}
Hence
$$\ew \,e^{\lambda\xi}=\ew\left[e^{\lambda\xi}\mathbbm{1}_{\{\xi\leq 0\}} \right]+\ew\left[e^{\lambda\xi}\mathbbm{1}_{\{\xi>0\}} \right]<\infty.$$
Now, writing $\xi=\xi^+-\xi^-$ (where $\xi^+:=\max\{\xi,0\}$, $\xi^{-}=\max\{-\xi,0\}$) and applying the fact that
\begin{equation}\label{e:known}
x+1\leq e^x \quad\text{for}\quad x\in\mathbb{R},
\end{equation}
we also get
$$\ew\xi^+=\ew\left[\xi\mathbbm{1}_{\{\xi>0\}} \right]=\frac{1}{\lambda} \ew\left[\lambda\xi\mathbbm{1}_{\{\xi>0\}} \right]\leq \frac{1}{\lambda}\ew\left[e^{\lambda\xi}\mathbbm{1}_{\{\xi>0\}} \right]<\infty,$$
which shows that $\ew \xi=\ew\xi^+-\ew\xi^-$ exists and that $\ew \xi<\infty$. We have therefore shown that \ref{cnd:d1} holds. To establish \ref{cnd:d2}, it suffices to observe that, if $\xi\geq -M$ a.s. for some $M>0$, then
$$\ew\,e^{\lambda|\xi|}=\ew\left[e^{-\lambda\xi}\mathbbm{1}_{\{\xi\leq 0\}}\right]+\ew\left[e^{\lambda \xi }\mathbbm{1}_{\{\xi> 0\}}\right]\leq e^{\lambda M}+\ew\left[e^{\lambda \xi }\mathbbm{1}_{\{\xi> 0\}}\right]<\infty,$$
since the second term on the right-hand side is finite due to \eqref{e:lem_d1}. Thus, in particular, it follows that
$$\ew|\xi|=\frac{1}{\lambda}\ew\lambda|\xi|\leq \frac{1}{\lambda}\ew\, e^{\lambda|\xi|}<\infty,$$
which completes the proof.
\end{proof} 

\begin{lemma}\label{lem:diac2}
Let $\{\xi_n\}_{n\in\n}$ be a~sequence of mutually independent and identically distributed random variables with values in $[-\infty, \infty)$, such that $\ew \xi_1$ exists and $-\infty \leq \ew \xi_1 < c$ for some $c \in \mathbb{R}$. Then there exists a~sequence $\{\bar{\xi}_n\}_{n \in \n}$ of mutually independent and identically distributed random variables, also taking values in $[-\infty, \infty)$, which is a.s. bounded from below (i.e., $\bar{\xi}_n\geq- K$ a.s. for all $n\in\n$ with some $K>0$) and satisfies
\begin{equation}\label{e:diac2_assert}
\ew\bar{\xi}_1 < c,\quad \xi_n \leq \bar{\xi}_n\;\;\text{a.s.}\quad \text{for all} \quad n \in \n.
\end{equation}
\end{lemma}
\begin{proof}
It is easy to check that 
$$\max\{\xi_1, -k \}=\xi^+-\min\{\xi_1^{-},k\}\quad\text{for each}\quad k\in\n.$$
Since $\min\{\xi_1^{-},k\} \uparrow \xi_1^{-}$ as $k\to\infty$, the Lebesgue monotone convergence theorem yields $\lim_{k\to\infty} \ew\left[\min\{\xi_1^{-},k\}\right]=\ew \xi_1^-$. Consequently, in view of the finiteness of $\ew \xi_1^+$, it follows that
$$\lim_{k\to\infty} \ew\left[\max\{\xi_1, -k \}\right]=\ew\xi^+-\ew\xi^-=\ew \xi<c.$$
We can therefore choose $K\in\n$ sufficiently large that $\ew\left[\max\{\xi_1, -K \}\right]<c$. Then the sequence $\{\bar{\xi}_n\}_{n \in \n}$ of random variables defined by
$$\bar{\xi}_n:=\max\{\xi_n, -K \}\quad\text{for}\quad n\in\n$$
fulfills the desired conditions.
\end{proof}

\begin{proof}[Proof of Lemma \ref{lem:diaconis}]
First of all, statement \ref{cnd:d1} of Lemma \ref{lem:diac1} guarantees that $\ew{e^{\lambda \xi_1}}<\infty$ for every $\lambda\in (0,\beta)$, and that $\ew\xi_1$ exists with $\ew\xi_1<\infty$. 

Fix an arbitrary $c>\ew\xi_1$. In view of Lemma \ref{lem:diac2}, we may assume, without loss of generality, that $\{\xi_n\}_{n\in\n}$ is a.s. bounded from below. Specifically, if this is not the case, we can replace $\{\xi_n\}_{n\in\n}$ with $\{\bar{\xi}_n\}_{n \in \n}$ that satisfies the conclusion of that lemma with the given~$c$. Then, owing to \eqref{e:diac2_assert}, whenever \eqref{e:diac_aasert} holds for $\{\bar{\xi}_n\}_{n\in\n}$, it also holds for $\{\xi_n\}_{n\in\n}$.

Given the assumption made above, it follows from statement \ref{cnd:d2} of Lemma \ref{lem:diac1} that $\ew\,e^{\lambda|\xi_1|}<\infty$ for every $\lambda\in (0,\beta)$, and $m:=\ew \xi_1\in\mathbb{R}$. For the sake of the subsequent analysis, let us fix any~$\gamma\in (0,\beta)$.

We will first prove that there exists $d>0$ such that
\begin{equation}\label{e:diac4}
\ew e^{\lambda\xi_1}\leq e^{\lambda m+\lambda^2 d}\quad\text{for any}\quad \lambda \in (0,\gamma].
\end{equation}
To this end, put $\xi:=\xi_1$ and define
$$d:=\frac{1}{\gamma^2}\left(\ew\, e^{\gamma|\xi|}-\gamma\ew|\xi|\right),$$
which is finite and, by \eqref{e:known}, positive ($\geq 1/\gamma^2$). Further, let $\lambda \in (0,\gamma]$ and note that
\begin{align*}
\frac{\gamma^2}{\lambda^2}\left|e^{\lambda\xi}-\lambda\xi-1\right|
&=\frac{\gamma^2}{\lambda^2}\left|\sum_{n=2}^{\infty}\frac{\lambda^n\xi^n }{n!} \right|
\leq \sum_{n=2}^{\infty} \frac{\gamma^2\lambda^{n-2}|\xi|^n}{n!}\leq \sum_{n=2}^{\infty} \frac{(\gamma|\xi|)^n}{n!}\\
&=\sum_{n=0}^{\infty} \frac{(\gamma|\xi|)^n}{n!} -\gamma|\xi|-1=e^{\gamma|\xi|}-\gamma|\xi|-1.
\end{align*}
Thus
$$\left|e^{\lambda\xi}-\lambda\xi-1\right|\leq \frac{\lambda^2}{\gamma^2}\left(e^{\gamma|\xi|}-\gamma|\xi| -1\right).$$
Consequently, we infer that
\begin{align*}
\ew\, e^{\lambda\xi}-\lambda m-1&\leq \ew\left|e^{\lambda\xi}-\lambda\xi-1 \right|\leq \frac{\lambda^2}{\gamma^2}\left(\ew\,e^{\gamma|\xi|}-\gamma\ew |\xi| -1\right)<\lambda^2 d,
\end{align*}
which gives
\begin{equation}\label{e:diac5}
\ew\, e^{\lambda\xi}\leq \lambda m+\lambda^2 d+1.
\end{equation}
Condition \eqref{e:diac4} follows now immediately from \eqref{e:diac5} and \eqref{e:known}. 

Let us now define
$$r(t):=e^{-(t-m)^2/(4d)}\quad\text{for}\quad t\in\mathbb{R},\quad\text{and}\quad c_0:=m+2d\gamma.$$
Clearly $r(t)\in (0,1)$ for every $t\neq m$, whence $r(c), r(c_0)\in (0,1)$ (since \hbox{$c,c_0>m$}).
To complete the proof, we will show that, for every $n\in\n$,
$$\pr\left(\sum_{j=1}^n \xi_j>nc\right)\leq r(c)^n\quad\text{if}\quad c\leq c_0,\quad\text{and}\quad \pr\left(\sum_{j=1}^n\xi_j>nc \right)\leq r(c_o)^n\quad\text{otherwise}.$$
For this aim, fix an arbitrary $n\in\n$ and put $\eta:=\sum_{j=1}^n\xi_j$. It suffices to consider the case where $c\leq c_0$, as otherwise $\pr(\eta>nc)\leq \pr(\eta>nc_0)$. Thus, assume that $c\leq c_0$, and let $\lambda:=(c-m)/(2d)$. Then $\lambda\in (0,\beta)$, since $0<c-m\leq c_0-m=2d\gamma<2d\beta$, and we have
\begin{equation}\label{e:diac6}
\lambda m+ \lambda^2d -\lambda c=-\lambda(c-m)+\lambda^2d = -\frac{(c-m)^2}{2d}+\frac{(c-m)^2}{4d}=-\frac{(c-m)^2}{4d}.
\end{equation}
Moreover, the mutual independence of $\xi_1,\xi_2,\ldots$, together with the fact that $\ew e^{\lambda\xi_1}<\infty$, yields
\begin{equation}\label{e:diac7}
\ew e^{\lambda\eta}=\prod_{j=1}^n \ew e^{\lambda\xi_j}=\left(\ew e^{\lambda\xi_1}\right)^n.
\end{equation} 
Finally, by applying the Chebyshev inequality, \eqref{e:diac7}, \eqref{e:diac4}, and then \eqref{e:diac6}, we get
\begin{align*}
\pr(\eta>nc)&=\pr(e^{\lambda\eta}>e^{\lambda nc})\leq e^{-\lambda n c}\,\ew e^{\lambda\eta}=\left(e^{-\lambda c} \ew e^{\lambda\xi_1}\right)^n\leq \left(e^{\lambda m+\lambda^2 d-\lambda c}\right)^n=r(c)^n,
\end{align*}
which establishes the desired claim.
\end{proof}

\section*{Acknowledgements}
The research of R.K. was supported by the Faculty of Applied Mathematics AGH UST statutory tasks within subsidy of the Polish Ministry of Science and Higher Education.

\section*{Funding}
This research did not receive any specific grant from funding agencies in the public, commercial, or not-for-profit sectors.

\section*{Data availability}
No data was used for the research described in the article.

\bibliographystyle{abbrv}
\small
\bibliography{ReferencesDatabase}

\begin{thebibliography}{10}

\bibitem{b:AlsmeyerIksanovRosler2009}
G.~Alsmeyer, A.~Iksanov, and U.~Rösler.
\newblock On distributional properties of perpetuities.
\newblock {\em Journal of Theoretical Probability}, 22(3):666--682, 2009.

\bibitem{b:BarnsleyDemko1985}
M.~F. Barnsley and S.~Demko.
\newblock Iterated function systems and the global construction of fractals.
\newblock {\em Proceedings of the Royal Society of London. Series A,
  Mathematical and Physical Sciences}, 399(1817):243--275, 1985.

\bibitem{b:BarnsleyDemkoEltonGeronimo1988}
M.~F. Barnsley, S.~G. Demko, J.~H. Elton, and J.~S. Geronimo.
\newblock Invariant measures for {Markov} processes arising from iterated
  function systems with place-dependent probabilities.
\newblock {\em Annales de l'Institut Henri Poincar{\'e}. Probabilit{\'e}s et
  Statistiques}, 24(3):367--394, 1988.
\newblock Erratum: volume 25 (1989), pages 589--590.

\bibitem{b:BuraczewskiDamekMikosch2016}
D.~Buraczewski, E.~Damek, and T.~Mikosch.
\newblock {\em Stochastic Models with Power-Law Tails: The Equation {$X = AX +
  B$}}.
\newblock Springer Series in Operations Research and Financial Engineering.
  Springer International Publishing, Cham, 1 edition, 2016.
\newblock XV + 320 pages.

\bibitem{b:DiaconisFreedman1999}
P.~Diaconis and D.~Freedman.
\newblock Iterated random functions.
\newblock {\em SIAM Review}, 41(1):45--76, 1999.

\bibitem{b:DubinsFreedman1966}
L.~E. Dubins and D.~A. Freedman.
\newblock Invariant probabilities for certain {Markov} processes.
\newblock {\em The Annals of Mathematical Statistics}, 37(4):837--848, 1966.

\bibitem{b:Dudley1966}
R.~M. Dudley.
\newblock Convergence of {Baire} measures.
\newblock {\em Studia Mathematica}, 27(3):251--268, 1966.
\newblock Correction published in Studia Mathematica 51 (1974), no. 3, p. 275.

\bibitem{b:EmbrechtsKluppelbergMikosch1997}
P.~Embrechts, C.~Klüppelberg, and T.~Mikosch.
\newblock {\em Modelling Extremal Events: For Insurance and Finance}, volume~33
  of {\em Stochastic Modelling and Applied Probability}.
\newblock Springer-Verlag, Berlin and Heidelberg, 1 edition, 1997.

\bibitem{b:FortetMourier1953}
R.~Fortet and {\'E}.~Mourier.
\newblock Convergence de la r\'epartition empirique vers la r\'epartition
  th\'eorique.
\newblock {\em Annales scientifiques de l'\'Ecole Normale Sup\'erieure},
  70(3):267--285, 1953.

\bibitem{b:GoldieMaller2000}
C.~M. Goldie and R.~A. Maller.
\newblock Stability of perpetuities.
\newblock {\em The Annals of Probability}, 28(3):1195--1218, 2000.

\bibitem{b:Hutchinson1981}
J.~E. Hutchinson.
\newblock Fractals and self-similarity.
\newblock {\em Indiana University Mathematics Journal}, 30(5):713--747, 1981.

\bibitem{b:Iksanov2016}
A.~Iksanov.
\newblock {\em Renewal Theory for Perturbed Random Walks and Similar
  Processes}.
\newblock Probability and Its Applications. Birkh{\"a}user, Cham, 1 edition,
  2016.
\newblock XIV + 250 pages.

\bibitem{b:KapicaSleczka2020}
R.~Kapica and M.~\'Sl\k{e}czka.
\newblock Random iteration with place dependent probabilities.
\newblock {\em Probability and Mathematical Statistics}, 40(1):119--137, 2020.

\bibitem{b:Karlin1953}
S.~Karlin.
\newblock Some random walks arising in learning models. {I}.
\newblock {\em Pacific Journal of Mathematics}, 3(4):725--756, 1953.

\bibitem{b:Lasota1995}
A.~Lasota.
\newblock From fractals to stochastic differential equations.
\newblock In P.~Garbaczewski, M.~Wolf, and A.~Weron, editors, {\em Chaos---The
  Interplay Between Stochastic and Deterministic Behaviour}, volume 457 of {\em
  Lecture Notes in Physics}, pages 235--255. Springer-Verlag, Berlin and
  Heidelberg, 1995.
\newblock Proceedings of the XXXI Winter School of Theoretical Physics,
  Karpacz, Poland, 1995.

\bibitem{b:LasotaMackey1999}
A.~Lasota and M.~C. Mackey.
\newblock Cell division and the stability of cellular populations.
\newblock {\em Journal of Mathematical Biology}, 38(3):241--261, 1999.

\bibitem{b:Wazewska}
A.~Lasota, M.~C. Mackey, and M.~Wa\.zewska-Czy\.zewska.
\newblock Minimizing therapeutically induced anemia.
\newblock {\em Journal of Mathematical Biology}, 13(2):149--158, 1981.

\bibitem{b:LasotaMyjak1998}
A.~Lasota and J.~Myjak.
\newblock Semifractals on {Polish} spaces.
\newblock {\em Bulletin of the {Polish} Academy of Sciences. Mathematics},
  46(2):179--196, 1998.

\bibitem{b:LasotaYorke1994}
A.~Lasota and J.~A. Yorke.
\newblock Lower bound technique for {Markov} operators and iterated function
  systems.
\newblock {\em Random \& Computational Dynamics}, 2(1):41--77, 1994.

\bibitem{b:Mendivil2015}
F.~Mendivil.
\newblock Time-dependent iteration of random functions.
\newblock {\em Chaos, Solitons \& Fractals}, 75:178--184, 2015.

\bibitem{b:MihocOnicescu1935}
G.~Mihoc and O.~Onicescu.
\newblock Sur les cha{\^\i}nes de variables statistiques.
\newblock {\em Bulletin des Sciences Math{\'e}matiques}, 59:174--192, 1935.

\bibitem{b:Norman1968}
M.~F. Norman.
\newblock Some convergence theorems for stochastic learning models with
  distance diminishing operators.
\newblock {\em Journal of Mathematical Psychology}, 5(1):61--101, 1968.

\bibitem{b:Stenflo1998}
{\"O}.~Stenflo.
\newblock Ergodic theorems for time-dependent random iteration of functions.
\newblock In M.~M. Novak, editor, {\em Fractals and Beyond: Complexities in the
  Sciences}, pages 129--136. World Scientific, Singapore, 1998.
\newblock Proceedings of the International Multidisciplinary Conference on
  Fractals, Valletta, Malta.

\bibitem{b:Stenflo2001}
{\"O}.~Stenflo.
\newblock Ergodic theorems for {Markov} chains represented by iterated function
  systems.
\newblock {\em Bulletin of the {Polish} Academy of Sciences: Mathematics},
  49(1):27--43, 2001.

\bibitem{b:Stenflo2002}
{\"O}.~Stenflo.
\newblock Uniqueness of invariant measures for place-dependent random
  iterations of functions.
\newblock In M.~F. Barnsley, D.~Saupe, and E.~R. Vrscay, editors, {\em Fractals
  in Multimedia}, volume 132 of {\em The IMA Volumes in Mathematics and its
  Applications}, pages 13--32. Springer-Verlag, New York, 2002.
\newblock Based on the IMA Annual Program on Mathematics in Multimedia,
  Minneapolis, Minnesota, 2001.

\bibitem{b:Szarek2003IM}
T.~Szarek.
\newblock Invariant measures for {Markov} operators with application to
  function systems.
\newblock {\em Studia Mathematica}, 154(3):207--222, 2003.

\bibitem{b:Szarek2003}
T.~Szarek.
\newblock Invariant measures for nonexpansive {Markov} operators on {Polish}
  spaces.
\newblock {\em Dissertationes Mathematicae}, 415:1--62, 2003.

\bibitem{b:Vervaat1979}
W.~Vervaat.
\newblock On a stochastic difference equation and a representation of
  non-negative infinitely divisible random variables.
\newblock {\em Advances in Applied Probability}, 11(4):750--783, 1979.

\bibitem{b:Werner2005}
I.~Werner.
\newblock Contractive {Markov} systems.
\newblock {\em Journal of the London Mathematical Society}, 71(1):236--258,
  2005.

\bibitem{b:Wojewodka2013}
H.~Wojew{\'o}dka.
\newblock Exponential rate of convergence for some {Markov} operators.
\newblock {\em Statistics \& Probability Letters}, 83(10):2337--2347, 2013.

\end{thebibliography}

\end{document}